\documentclass{amsart}

\numberwithin{equation}{section}

\usepackage{amssymb}
\usepackage{enumerate, xspace}
\usepackage{dsfont}
\usepackage{lipsum}
\usepackage[normalem]{ulem}
\usepackage[colorlinks]{hyperref}

\newtheorem{thm}{Theorem}[section]
\newtheorem{lem}[thm]{Lemma}

\newtheorem{prop}[thm]{Proposition}

\newtheorem{defin}{Definition}

\newtheorem{rem}[thm]{Remark}

\newcommand\cB{{\mathcal B}}
\newcommand\cC{{\mathcal C}}

\newcommand\cF{{\mathcal F}}
\newcommand\cG{{\mathcal G}}
\newcommand\cH{{\mathcal H}}

\newcommand\cK{{\mathcal K}}
\newcommand\cL{{\mathcal L}}
\newcommand\cM{{\mathcal M}}
\newcommand\cO{{\mathcal O}}
\newcommand\cR{{\mathcal R}}
\newcommand\cT{{\mathcal T}}

\newcommand\cW{{\mathcal W}}

\newcommand\bC{{\mathbb C}}

\newcommand\bF{{\mathbb F}}

\newcommand\bN{{\mathbb N}}
\newcommand\bR{{\mathbb R}}

\newcommand\ve{\varepsilon}

\newcommand\vf{\varphi}

\newcommand{\Const}{C_\sharp}
\newcommand\czero{c_0}
\newcommand\Czero{C_0}

\newcommand{\jj}{\text{\it\bfseries j}}

\newcommand{\axp}{\beta}
\renewcommand\div{{\operatorname {div}\,}}
\newcommand\tr{{\operatorname{Trace}}}
\newcommand\vfp{\vartheta}
\newcommand\Id{{\mathds{1}}}

\newcommand{\srk}{\eta}
\newcommand{\sri}{\theta}
\newcommand{\param}{{ \varpi}}
\newcommand{\Lb}{L}

\begin{document}

\title[Anosov diffeomorphisms]{Anosov diffeomorphisms, anisotropic BV spaces and regularity of foliations}
\author{Wael Bahsoun}
\address{Department of Mathematical Sciences, Loughborough University,
Loughborough, Leicestershire, LE11 3TU, UK}
\email{W.Bahsoun@lboro.ac.uk}
\author{Carlangelo Liverani}
\address{Carlangelo Liverani\\
Dipartimento di Matematica\\
II Universit\`{a} di Roma (Tor Vergata)\\
Via della Ricerca Scientifica, 00133 Roma, Italy.}
\email{{\tt liverani@mat.uniroma2.it}}
\thanks{WB and CL would like to thank the Heilbronn Institute
for Mathematical Research where part of this work took place. WB would also like to thank the Department of Mathematics, University of Rome Tor Vergata for its hospitality where some of this work was carried. CL would like to thank Viviane Baladi, Sebastien Gou\"ezel and Mark Demers for helpful discussions. CL acknowledges the MIUR Excellence Department Project awarded to the Department of Mathematics, University of Rome Tor Vergata, CUP E83C18000100006. The paper was partially supported by the Grant PRIN 2017S35EHN}

\begin{abstract}
Given any smooth Anosov map we construct a Banach space on which the associated transfer operator is quasi-compact.
The peculiarity of such a space is that in the case of expanding maps it reduces exactly to the usual space of functions of bounded variation which has proven particularly successful in studying the statistical properties of piecewise expanding maps.
Our approach is based on a new method of studying the absolute continuity of foliations which provides new information that could prove useful in treating hyperbolic systems with singularities.
\end{abstract}
\maketitle

\section{Introduction}
Starting with the paper \cite{BKL}, there has been a growing interest in the possibility to develop a functional analytic setting allowing the direct study of the transfer operator of a hyperbolic dynamical system.
The papers \cite{GL, GL1, BT, BT1, Ba1, Ba2, Ba3, Ba4, T1} have now produced quite satisfactory results for the case of Anosov diffeomorphisms (or, more generally, for uniformly hyperbolic basic sets). Important results, although the theory is not complete yet, have been obtained for flows  \cite{L, BuL, BuL2, GLP, FT2, DyZ, D17}, group extensions and skew products (\cite{F11, AGT}). Moreover, recently a strong relation with techniques used in semiclassical analysis (e.g. see \cite{FR, FRS, FT1, FT2, DyZ}) has been unveiled.  Also, one should mention the recent discovery of a deep relation with the theory of renormalization of parabolic systems \cite{GL19}.  In addition, such an approach has proven very effective in the study of perturbation of dynamical systems \cite{KL1,KL3} and in the investigation of limit theorems \cite{G10}.
At the same time \cite{KL2, KL4} have shown that this strategy can be extended to a large class of infinite dimensional systems (coupled map lattices), but limited to the case of piecewise expanding maps. However, there has been no progress in applying it to coupled lattices of Anosov maps when the coupling introduces discontinuities in the system.\footnote{The only available results are restricted to a special class of interactions that salvage structural stability \cite{PS}.}  Moreover, only partial progress has been accomplished in extending such an approach to partially hyperbolic maps \cite{T2} and to piecewise smooth uniformly hyperbolic systems \cite{DL, BG, BG2, DZ1, DZ2, DZ3,  BaL, BDL}.   The recent book \cite{Ba6} provides an extensive account and a thorough illustration of the topic.

The present paper is motivated by the current shortcomings in the applications of the functional analytic strategy to piecewise smooth hyperbolic maps. Indeed, while in two dimensions the approach can be applied to a large class of systems \cite{DZ2, DZ3}, in higher dimensions it is limited to the case in which the map is well behaved up to and including the boundary \cite{BG, BG2} or some special skew product cases \cite{G1,G2}.  In the case of piecewise expanding maps the latter problems are dealt with by using different Banach spaces. In particular, a huge class of piecewise expanding maps can be treated by using the space of functions of bounded variation ($BV$) or their straightforward generalisations \cite{Sa, Li13a, Bu13, Li13b}. It is thus natural to construct Banach spaces that generalize BV and are adapted to the study of the transfer operator associated  with hyperbolic maps.

Bounded variation like spaces could allow to extend the known results to higher dimensional, possibly infinite dimensional (coupled Anosov map lattices), systems. Also they could allow to treat higher dimensional hyperbolic maps with strong singularities (e.g. billiards). In addition, such spaces should be useful to investigate numerically the spectrum of the transfer operator via Ulam-type perturbations, which proved to be very successful when dealing with expanding maps and $BV$ functions \cite{Liv}. Indeed previous investigations of Ulam approximation for  Anosov systems left several questions unanswered due to the inadequacy of the Banach spaces used, e.g. \cite{BKL}. 

Unfortunately, none of the Banach spaces proposed in the literature for the study of the transfer operator associated with general Anosov diffeomorphisms, or general piecewise Anosov, reduces exactly to $BV$ when the stable direction is absent.

The purpose of this paper is to correct this state of affairs by introducing a template for Banach spaces with the above property. We will apply it to the case of smooth Anosov diffeomorphisms. Although for such examples this provides limited new information, it shows that the proposed space is well adapted to the hyperbolic structure, hence there is a concrete hope that this space can be adapted to study general piecewise Anosov maps and Anosov coupled map lattices in a unified setting. A substantial amount of work is still needed to find out if such a hope has some substance or not. Nevertheless, the present arguments are worth presenting since they are remarkably simple and natural.

An additional fact of interest in the present paper is the characterization of invariant foliations and, more generally, the method used to study the evolution of foliations under the dynamics. It is well known that the stable foliation is only H\"older, although the leaves of the foliations enjoy the same regularity as the map. Nevertheless, a fundamental discovery by Anosov is that the holonomy associated to the foliation is absolutely continuous and the Jacobian is H\"older. The establishing of this fact is not trivial and, especially in the discontinuous case, entails a huge amount of work \cite{KS}. Here we show that the properties of such foliations can be characterized infinitesimally, hence considerably simplifying their description, see Definition \ref{def:foliation}. In particular,  given a foliation $F$, the Jacobian $J^F$ of the associated Holonomy can be seen as a quantity produced by a flow, of which we control the generator $H^F$. See Lemma \ref{lem:holo} for a precise explanation of this fact. We believe this point of view will be instrumental in treating discontinuous maps.

The structure of the paper is as follows: Section \ref{sec:norms} contains the definition of the Banach space and the statement of the main theorem (Theorem \ref{thm:main}). Section \ref{sec:lasota} contains the usual Lasota-Yorke estimate, while section \ref{sec:comp} contains the estimate on the essential spectrum of the operator. Section \ref{sec:top} contains some comments on the peripheral spectrum. Appendix \ref{sec:normscq} reminds the reader of some convenient properties of $\cC^r$ norms. Appendix \ref{app:fol} establishes various properties of the foliations of Anosov maps that should be folklore among experts, but we could not locate anywhere (in particular the smoothness of the Jacobian of the stable holonomy along stable leaves). Moreover, as previously mentioned, such properties are expressed totally in local terms, contrary to the usual approach. Finally, Appendix \ref{sec:test} contains a few technical estimates on the test functions.

{\bf Notation.} In this paper we will use $\Const$ to designate a constant that depends only on the map $T$ and on the choice of coordinates, but whose actual value is irrelevant to the tasks at hand. Hence, the value of $\Const$ can change from one occurrence to the next and it is determined by the equation in which it appears. Analogously, we will use $C_{a,b,\dots}, c_{a,b,\dots}$ for generic constants that depend also on the quantities $a,b,\dots$.

\section{The Banach space}\label{sec:norms}
Our goal is to develop a space in the spirit of $BV$ for the study of the statistical properties of a dynamical system  $(M,T,\mu)$ where $M$ is a compact $\cC^r$ manifold, $T$ is uniformly hyperbolic and $\mu$ is the SRB measure.\footnote{ Of course, there are many other functional spaces to analyse such maps (e.g. see \cite{Ba6}), however we restrict to this class of maps to illustrate the construction of the space in the simplest possible form.} Let us be more precise.
\subsection{The phase space}
Let $r\ge 2$ be an integer and $M$ be a $\cC^r$ $d$-dimensional compact manifold where the differentiable structure is the one induced by the atlas $\{V_i, \phi_i\}_{i=1}^S$, $V_i\subset M$, $S\in \bN$. To be more precise we consider a fixed smooth partition of unity $\{\vfp_i\}$ subordinated to the cover $\{V_i\}$. We then define a smooth volume form $\omega$ by
\begin{equation}\label{eq:form}
\int_Mh\; d\omega=\sum_{i=1}^S\int_{U_i}h\circ  \phi_i^{-1}(z)\;\vfp_i\circ \phi_i^{-1}(z) d z,
\end{equation}
where $U_i:=\phi_i(V_i)$. From now on all integrals will be with respect to such a form although we will not specify it explicitly. 

\subsection{The map and the cones}\label{sec:map} We consider an Anosov diffeomorphism $T\in\operatorname{Diff}^r(M)$. That is, there exists $\lambda>1, \nu\in (0,1)$, $\czero\in(0, 1)$ and a continuous cone field (stable cone) $\cC=\{C(\xi)\}_{\xi\in M}$, $ \overline{C(\xi)}=C(\xi)\subset T_\xi M$ such that $D_{\xi}T^{-1} C(\xi)\subset \text{int}(C(T^{-1}(\xi)))\cup\{0\}$ and\footnote{ Here the norm is defined by some smooth Riemannian structure, the actual choice of such a structure will be irrelevant in the following, it will just affect the constants.}
\begin{equation}\label{eq:anosov}
\begin{split}
&\inf_{\xi\in M}\inf_{v\in C(\xi)} \|D_\xi T^{-n}v\|> \czero\nu^{-n}\|v\|\\
& \inf_{\xi\in M}\inf_{v\not\in C(\xi)} \|D_\xi T^{n}v\|> \czero\lambda^{n}\|v\|.
\end{split}
\end{equation}
In higher dimensions a cone may have many geometric shapes. It is convenient, and useful, to ask that they be subsets $\cK$ of the Grassmannian. More precisely, we can assume, without loss of generality, that, for each $\xi\in V_i$ and calling $\cM(d_u, d_s)$ the set of $d_u\times d_s$ matrices, 
\begin{equation}\label{eq:cones}
\begin{split}
\cK_\theta&=\{U\in \cM(d_u, d_s)\;:\; \|U\|\leq \theta\}\\
D \phi_i C(\xi)&=\{(x,y)\in\bR^{d_u}\times\bR^{d_s}=\bR^d\;:\; x=Uy,\, U\in \cK_1\}\\
&=\{(x,y)\in\bR^{d_u}\times\bR^{d_s}=\bR^d\;:\; \|x\|\leq\|y\|\},
\end{split}
\end{equation}
where $U$ is any $d_u\times d_s$ matrix. Then the strict cone field invariance reduces to the existence of $\srk\in(0,1)$ such that
\[
\begin{split}
&D\phi_j DT^{-1}C(\xi) \subset\{(x,y)\in\bR^{d_u}\times\bR^{d_s}=\bR^d\;:\; \|x\|\leq \srk \|y\|\}\\
&D\phi_{j'} DTC_c(\xi)\subset\{(x,y)\in\bR^{d_u}\times\bR^{d_s}=\bR^d\;:\; \|y\|\leq \srk\|x\| \}
\end{split}
\]
where $V_j\ni T^{-1}(\xi)$, $V_{j'}\ni T(\xi)$ and $C_c(\xi)=\overline{T_\xi M\setminus C(\xi)}$.

\subsection{ Transfer Operator}\label{sec:TO}

We are interested in studying the statistical properties of the above systems. One key tool used to such an end is the Transfer Operator:
for each $h\in\cC^{1}$ we define\footnote{ By $\det$ we mean the density of $T^*\omega$ with respect to $\omega$.}
\begin{equation}\label{eq:pfr}
\cL h=\left[h\cdot|\det(DT)|^{-1}\right]\circ T^{-1}.
\end{equation}

Accordingly, for each $n\in \bN$,
\[
\int_M \vf \cL^n h=\int_M h \vf\circ T^n.
\]
It is then clear that the behaviour of the integrals on the left of the above equation can be studied if one understands the spectrum of $\cL$. Obviously such a spectrum depends on the space on which the operator is defined. Several proposals have been developed to have spaces on which $\cL$ is quasi-compact. Such proposals are extremely effective when the map is smooth, see \cite{Ba6} for a review, less so for discontinuous systems. Since in the case of expanding maps $BV$ is very effective \cite{Li13b}, it is natural to investigate if one can construct a space, suitable for the study of invertible maps, that reduces to $BV$ when the stable direction is absent. In the next sections we define Banach spaces $\cB^{0,q}$ and $\cB^{1,q}$ that, when the stable direction is absent, reduce to $L^1$ and $BV$ respectively (see Remark \ref{rem:no-stable}). Although we do not discuss discontinuous maps, this is certainly a first step to develop a viable alternative to the current approaches.
To show that the space is potentially well behaved we prove the following Theorem.
\begin{thm}\label{thm:main}  For each $q\in\{1, \dots, r-2\}$,\footnote{Note that this condition is non vacuous only if $r\geq 3$. However, we stress that all the results related to the regularity of foliations in Appendix \ref{app:fol} hold true for $r\ge 2$.} the operator $\cL$ has an extension,\footnote{ We still call such an extension $\cL$.} which belongs to $L(\cB^{0,q},\cB^{0,q})$ and $L(\cB^{1,q},\cB^{1,q})$; moreover,
\begin{enumerate}
\item $\cL: \cB^{1,q}\to\cB^{1,q}$ is a quasi-compact operator with spectral radius $1$ and essential spectral radius $\sigma_{ess}:=\max\{\lambda^{-1},\nu\}$.
\item The peripheral spectrum of $\cL$ consists of finitely many finite groups; in particular $1$ is an eigenvalue. 
\item Setting $h_*:=\Pi_11$, where $\Pi_1$ is the spectral projection of $\cL$ associated with the eigenvalue $1$, then $h_*$ is canonically associated to a measure whose ergodic decomposition  corresponds to the spectral decomposition for the Anosov map and consists of the physical measures.
\end{enumerate} 
\end{thm}

\begin{proof}
The proof of part (1) can be found in Lemma \ref{lem:ess}. The proof of (2) is given in Lemma \ref{lem:peripheral}. Finally, the canonical correspondence of $h_*$ with a distribution, mentioned in point (3), is detailed in Lemma \ref{lem:embedding} (see also Remark \ref{rem:identify} for the use of such a canonical correspondence in this paper) while in Lemma \ref{lem:peripheral} it is showed that the associated distribution is, in fact, a measure. Finally, the proof of (3) is provided by Lemma \ref{lem:spectra}.
\end{proof}

Observe that $BV\subset \cB^{1,q}$, see Remark \ref{rem:BV}, hence the above theorem implies that for Anosov maps the spectrum $\sigma_{\cB^{1,q}}(\cL)$ determines the decay of correlation for $BV$ densities. In particular, consider a transitive Anosov map. Then $h_*$ is ergodic and corresponds to the unique $SRB$ measure $\mu_{{\scriptscriptstyle SRB}}$. Also, Theorem \ref{thm:main} implies that, for all $\theta> \sigma_{ess}$, there exists a constant $C_\theta>0$, finitely many eigenvalues $\{\theta_j\}$, $|\theta_j|\in (\theta,1)$, and  finite rank operators $B_j:\cB^{1,1}\to \cB^{1,1}$, with spectral radius equal one,  such that, for all $\vf\in \cC^{2}$ and  rectifiable sets $A$:\footnote{ By rectifiable we mean that $\Id_A\in BV$.}
\[
\left|\int_M \vf\circ T^n \Id_A-\int_M \vf \mu_{{\scriptscriptstyle SRB}}\int_M \Id_A-\sum_j\theta_j^n \int_M\vf B_j^n \Id_A\right| \leq C\|\vf\|_{\cC^{2}}\|\Id_A\|_{BV}\theta^n.
\]
Note that a similar estimate could be obtained using the spectral properties on  spaces already existing in the literature and deducing the behaviour for $BV$ densities by an approximation argument. However, this would produce a less sharp result (in particular, it would allow only  $\theta>\sigma_{ess}^\alpha$, for some $\alpha<1$).
 In addition, the following are direct consequences of Theorem \ref{thm:main}:\\
$\bullet$ The Central Limit Theorem and other statistical properties for observables that are multipliers of $BV$ via the usual spectral approach of analytic perturbation theory, e.g. see \cite{G15}.\\
$\bullet$  Statistical aspects of random perturbations: let $T_0$ be a transitive Anosov map. Let $B_{T_0}$ be a sufficiently small neighbourhood of $T_0$ in the $C^1$-topology so that condition \eqref{eq:anosov} is satisfied for all $T\in B_{T_0}$ with uniform constants. Let $\Xi:=\sup_l\sum_{k}\| \left[\partial_l (DT_0^{-1})_{l,k}\right]\|_{\cC^1}$ and define the following family of maps 
$$G_\Xi=\{T\in C^2(M):\, T\in B_{T_0} \text{ and } \sup_l\sum_{k}\| \left[\partial_l (DT^{-1})_{l,k}\right]\|_{\cC^1}\le 2\Xi\}.$$ 
One can study, for instance, iid compositions with respect to some product probability measure $\mathbb P$ defined on on $G_{\Xi}^{\mathbb N}$. Spectral properties of the annealed transfer operator associated with the above random map follows from this work and  stability results can be obtained using the current setting and the framework of \cite{KL1}.\\

\subsection{Foliations}\label{sec:foliation} A fundamental ingredient in the understanding of hyperbolic maps is the study of dynamical foliations, hence a small digression is in order.
\begin{defin} \label{def:foliation0}
A $\cC^r$ $t$-dimensional foliation $W$ is a collection $\{W_\alpha\}_{\alpha\in A}$, for some set $A$, such that the $W_\alpha$ are pairwise disjoint, $\cup_{\alpha\in A}W_\alpha=M$ and for each $\xi\in W_\alpha$ there exists a neighborhood $B(\xi)$ such that the connected component of $W_\alpha \cap B(\xi)$ containing $\xi$, call it $W(\xi)$, is a $\cC^r$ $t$-dimensional open submanifold of $M$. We will call $\cF^r$ the set of $\cC^r$ $d_s$-dimensional foliations.
\end{defin}

\begin{defin}\label{def:F}
A foliation $W$ is adapted to the cone field $\cC$ if, for each $\xi\in M$, $T_\xi W(\xi)\subset C(\xi)$. Let $\cF^r_\cC$ be the set of $\cC^r$ $d_s$-dimensional foliations adapted to $\cC$.
\end{defin}

Given a $d_s$-foliation adapted to $\cC$ we can associate to it local coordinates as follows. Let $\delta_0>0$ be sufficiently small so that for each $\xi\in M$ there exists a chart $(V_i,\phi_i)$ with $\xi\in V_i$ and such that $U_i:=\phi_i(V_i)$ contains the ball $B_{\delta_0}(\phi_i(\xi))$.\footnote{Here, and in the following, we use $B_\delta(x)$ to designate $\{z\in\bR^{d'}\:\;: \|x-z\|\leq \delta\}$ for any $d'\in\bN$.} Also, choose $U^0=U^0_u\times U^0_s\subset \bR^{d_u}\times\bR^{d_s}$ with $U^0_u= B_{\delta_0/2}(0)$, $U^0_s= B_{\delta_0/2}(0)$. Next, for each $z\in U_i$, let $W(z)$ be the connected component of $\phi_i(W)$ containing $z$.\footnote{ Refer to Definition \ref{def:foliation0} for the exact meaning of ``connected component". Also note the abuse of notation since we use the same name for the sub-manifond in $M$ and its image in the chart.} Define the function $F_\xi:U^0\to\bR^{d_u}$ by $\{ (F_\xi(x,y)+x_\xi,y_\xi+y)\}=\{(w,y+y_\xi)\}_{w\in \bR^{d_u}}\cap W(x+x_\xi,y_\xi)$, where $(x_\xi,y_\xi)=\phi_i(\xi)$.\footnote{ The fact that the intersection is non void and consists of exactly  one point follows trivially from the fact that the foliation is adapted to the cone field, hence the two manifolds are transversal.}
That is, $W(x+x_\xi,y_\xi)$ is exactly  the graph of the function $F_\xi(x,\cdot)+x_\xi$. Moreover, 
 \begin{equation}\label{eq:F-def}
 F_\xi(x,0)=x.
 \end{equation}
In addition, we ask $\delta_0$ to be small enough that the expression of $DT$ in the above charts is roughly constant. See Lemma \ref{eq:S-est} and its poof for the precise condition.

\begin{rem} \label{rem:foliation} The above construction defines the triangular coordinates $\bF_\xi(x,y)=(F_\xi(x,y),y)$ which describes locally the foliation. In fact, $(\phi_i^{-1}\circ \bF_\xi(U^0), \bF_\xi^{-1}\circ \phi_i)$ is a local chart of $M$ in which the foliation is trivial (the leaves are all parallel).
In the following we will often use such coordinates without mention if it will not create confusion. Also, to ease notation, we will confuse $V_i$ with $\phi_i(V_i)$ when not ambiguous. In addition, we will use $\bF$ to indicate the collection of maps $\{\bF_{\xi}\}$ and the same for $F$. Of course, $\bF$ is not unique, since we can chose different charts for the same $\xi$, however different choices are equivalent so we assume that some choice has been made. Clearly $\bF$ defines uniquely $W$.
\end{rem}
\begin{defin}\label{def:foliation}
For each $r\in\bN$, $r\geq 2$, and $L>0$,   let\footnote{Note that equation \eqref{eq:cones}, Definition \ref{def:F} and the subsequent description imply that $\| \partial_yF_\xi(x,y)\|\leq 1$, hence, recalling \eqref{eq:F-def}, $\|F_\xi\|_\infty\leq \delta_0$. Accordingly, in the first line of the definition of $\cW_L^r$ the cases $k\in \{0,1\}$ are superfluous.}
\[
\begin{split}
&\overline{\cF}^r_\cC:=\left\{W\in\cF^r_\cC\;:\;  \bF \in\cC^r(U^0,\bR^{d})\right\}\\
&\cW_L^r:=\Big\{W\in\overline{\cF}^r_\cC\;:\;  \sup_\xi \sup_{x\in U^0_u} \sup_{|\alpha|=k}\|\partial^\alpha_y F_\xi(x,\cdot)\|_{\cC^0(U^0_s,\bR^{d_u})}\leq L^{(k-1)^2} , 2\leq k\leq r;\\
&\phantom{\cW_D^r:=\Big\{W\in\overline{\cF}^r_\cC\;:\;}
 \sup_\xi \sup_{x\in U^0_u} \sup_{|\alpha|=k}\|\partial^\alpha_y  H^{F_\xi}(x,\cdot)\|_{\cC^{0}(U^0_s,\bR^{d_s})}\leq L^{(k+1)^2}, 0\leq k\leq r-2\Big\}, 
\end{split}
\]
where 
\begin{equation}\label{eq:HF}
\begin{split}
 H^{F_{\xi}}(x,y)&=\sum_{j=1}^{d_u}\left[\partial_{x_j}\left(\left[\partial_y (F_{\xi})_j\right]\circ \bF_\xi^{-1}\right)\right]\circ \bF_{\xi}(x,y)\\
 &=\sum_{ij}\partial_{x_i}\partial_y (F_{\xi})_j\cdot (\partial_x F_{\xi})^{-1}_{ij}.
\end{split}
\end{equation}
\end{defin}

\begin{rem}
Since the invariant foliation is not $\cC^r$ (in general it is only H\"older, although it consists of  $\cC^r$ leaves) it does not belong to $\cW^r_L$ for any $L$. Yet, it belongs to its closure, if $L$ is large enough (see Remark \ref{rem:invariant}).
\end{rem}

\begin{rem} Note that the functions $H^F$ are related to the Jacobian of the stable holonomy (see Lemma \ref{lem:holo}), hence it does not make sense to require them to be uniformly smooth. In general it is possible to control effectively only their H\"older norm, yet, restricted to the stable direction they turn out to be smooth. Indeed, this is the whole content of Appendix \ref{app:fol}. 
\end{rem}
\begin{rem}
The role of $H^F$ in the definition of $\cW_L^r$ will become apparent in the proof of the Lasota-Yorke inequality in Proposition \ref{lem:norm}, namely in \eqref{eq:H_use}. Hence, controlling the $\sup_\xi \sup_{x\in U^0_u}\| H^{F^n_\xi}(x,\cdot)\|_{\cC^{k}(U^0_s,\bR^{d_s})}$, uniformly in $n$, is essential.
\end{rem}
Next we would like to define the evolution of a foliation $W\in \cW^r_L$ under $T$.  Let $W^n:=T^{-n}W:=\{T^{-n}W_\alpha\}_{\alpha\in A}$. Clearly $W^n\in \overline{\cF}_\cC^r$, but much more is true.
\begin{lem} \label{lem:foliation}
There exists $n_0\in\bN$ and $L>0$ such that for all $n\in \bN$, $n\geq n_0$, $L_1\geq L$ and $W\in \cW^r_{L_1}$, we have $W^n\in  \cW^r_{L_1  /2}$.
\end{lem}

\begin{rem} By considering an appropriate power of the map, rather than the map itself, we can always reduce to the case $n_0=1$. We will do exactly this in the following.
\end{rem}
\begin{rem} From now on $L$ is fixed so that Lemma \ref{lem:foliation} holds true. Since the choice of $L$ depends only on $T$ and $M$, in the future we will not make the $L$ dependence explicit in the constants.
\end{rem}
Lemma \ref{lem:foliation} is proved in Appendix \ref{app:fol}. In fact we prove the more general Proposition \ref{lem:foliationg} which implies Lemma \ref{lem:foliation} (see Remark \ref{rem:tauzero}).

\subsection{Test Functions}\label{sec:testf}
Since we will want to be free to work with high order derivatives, it is convenient to choose a norm $\|\cdot\|_{\mathcal{C}^{\rho} }$, $\rho\in \bN\cup\{0\}$, equivalent to the standard one, for which $\cC^\rho$ is a  Banach Algebra. We thus define the weighted norm in $\cC^\rho(M,\cM(m,n))$, where $\cM(m,n)$ is the set of the $m\times n$ (possibly complex valued) matrices,  
\begin{equation}
\label{def of the C norm}
\begin{split}
&\|\varphi \|_{\mathcal{C}^0}=\sup_{x \in M}\sup_{i\in\{1, \dots,n\}}\sum_{j=1}^m|\varphi_{i,j}(x)| \\
&\|\varphi\|_{\mathcal{C}^{\rho}}=\sum_{k=0}^{\rho}\param^{\rho-k} \sup_{|\alpha|=k}\|\partial^{\alpha}\varphi\|_{\mathcal{C}^0},
\end{split}
\end{equation}
where, $\param\geq 2$ is a parameter to be chosen later (see \eqref{eq:param-sel}), $\alpha$ is a multi-index  $\alpha=(\alpha_1,\cdots, \alpha_d)$ with $\alpha_i\in \bN\cup\{0\}$, we denote $|\alpha|=\sum_{i=1}^d\alpha_i$, and $\partial^{\alpha}=\partial_{x_1}^{\alpha_1}\cdots\partial_{x_d}^{\alpha_d}$.

Note that the above definition implies
\begin{equation}\label{eq:norm-q-def}
\|\varphi\|_{\mathcal{C}^{\rho+1}}=\param^{\rho+1} \|\varphi\|_{\mathcal{C}^{0}}+\sup_{i}\|\partial_{x_i} \varphi\|_{\mathcal{C}^{\rho}}.
\end{equation}
 The next Lemma is proven in Appendix \ref{sec:normscq}.
\begin{lem}
\label{properties of the C norm}
For every $ \rho,n, m,s  \in\bN$, $ \psi\in \cC^\rho(M,\cM(m,n))$ and $\vf\in \cC^\rho(M,\cM(m,s))$ we have 
\[
\|\varphi\psi\|_{\mathcal{C}^{\rho}}\le \|\varphi\|_{\mathcal{C}^{\rho}}\|\psi\|_{\mathcal{C}^{\rho}}.
\]
Moreover if $\vf\in \cC^\rho(M,\cM(m,n))$ and $\psi\in\cC^\rho(M,M)$, then\footnote{ $\cC^\rho(M,M)$ is defined in the natural manner using the  norm \eqref{def of the C norm} in the charts $(V_i,\phi_i)$, also we use the charts to identify $T_x M$ with $\bR^d$, hence $D \psi\in \cM(d,d)$.}
\[
\|\varphi\circ \psi\|_{\cC^\rho}\le \sum_{k=0}^{\rho} {{\rho}\choose{k}} \param^{\rho-k}\|\varphi\|_{\cC^{k}}\prod_{i=1}^{k}\|(D\psi)^t\|_{\cC^{\rho-i}}. 
\]

\end{lem}
\begin{defin}\label{def:semi-norm}
For each $\vf\in\cC^{r}(M, \bC^l)$ and $W\in\cF^r_{\cC}$ let $\vf_{\xi,x}(\cdot)=\vf\circ \phi_{i}^{-1}\circ \bF_{\xi}(x,\cdot)$, $q\le r$, and define\footnote{We use the standard notation $\vf_{\xi,x}=((\vf_{\xi,x})_1,\dots, (\vf_{\xi,x})_l)$.}
\begin{equation}\label{eq:norm_dual}
\| \vf\|_q^W:=\sup_{\xi\in M}\sup_{x\in U^0_u}\|\vf_{\xi,x}\|_{\cC^q(U^0_s, \bC^{l})}=\sup_{\xi\in M}\sup_{x\in U^0_u}\sum_{j=1}^l\|(\vf_{\xi,x})_j\|_{\cC^q(U^0_s, \bC)}.
\end{equation}

\end{defin}
\begin{rem} It is easy to verify that a different choice of the charts produces a uniformly equivalent class of norms.
\end{rem}


\subsection{A class of measures} To be precise, we are going to define a Banach space of distributions. We will be interested in measures that belong to such a space.
Define
\begin{equation}\label{eq:dual}
\Omega_{L,q,l}=\left\{(W,\vf)\in \cW_L^r\times \cC^q(M, \bC^l)\;:\;\|\vf\|_q^W\leq 1\right\}
\end{equation}

and lift the dynamics to $\Omega_{L,q,l}$ by $T_*(W,\vf)=(T^{-1}W,\vf\circ T)$.

\begin{lem}\label{lem:dual}
 For each  $\sigma\in (\nu,1)$, there exists constants $A_0,B_0>0$ such that, for each choice of $\param\geq 2$, each $(W,\vf)\in\Omega_{L,q,l}$, $q\in\{0,\dots, r-1\}$ and $n\in\bN$
\[
\begin{split}
&\|\vf\circ T^n\|^{T^{-n}W}_q\leq A_0\|\vf\|^{W}_q;\\
&\|\vf\circ T^n\|^{T^{-n}W}_{q+1}\leq { A_0 \sigma^{nq}\|\vf\|^{W}_{q+1}+B_0}\|\vf\|^{W}_q.
\end{split}
\]
\end{lem}
The Lemma is proved in Appendix \ref{sec:test}.
Note that Lemmata \ref{lem:foliation} and \ref{lem:dual} imply $T_*\Omega_{L, q,l}\subset \Omega_{L, q,l}$.

It is now time to define the norms. Given a function $h\in\cC^1(M,\bC)$ we define\footnote{As already remarked the differential structure and the volume form are defined via the charts, thus, to be precise,
\[
\int_M h\,\div \vf=\sum_{i=1}^S\int_{U_i}h\circ \phi_i^{-1}(z)\vartheta_i\circ\phi_i^{-1}(z)\sum_{j=1}^d(\partial_{z_j}[\vf_j\circ \phi_i^{-1}])(z)\;dz.
\]
\label{foo:dive}}
\begin{equation}\label{eq:norms}
\begin{split}
&\|h\|_{0,q}:=\sup_{(W,\vf)\in\Omega_{L,q,1}}\left |\int_M h \, \vf \right |\\
& \|h\|^*_{1,q}:=\sup_{(W,\vf)\in\Omega_{L,q+1,d}}\left |\int_M  h \, \div \vf\right |\\
&\|h\|^-_{1,q}:=a\|h\|_{0,q}+\|h\|^*_{1,q},
\end{split}
\end{equation}
for any $q\in\bN\cup\{0\}$ and some fixed $a>0$ to be chosen later (see Proposition \ref{lem:norm}).

We are then ready to define the Banach spaces. The space $\cB^{0,q}$ is the Banach spaces obtained by completing $\cC^1(M,\bR)$ in the $\|\cdot\|_{0,q}$ norm.\footnote{ The completion can be achieved within the space of distributions of order $q$.} We are not interested in making the same choice for the norm $\|\cdot\|^-_{1,q}$ since this, in the case of $d_s=0$ would yield the Sobolev space $W^{1,1}$ rather than the space of function of bounded variations that we are interested in. We use thus the analogous of the standard procedure to define $BV$ starting from $W^{1,1}$. First let us define the new norm, for each $h\in\cB^{0,q}$,
\begin{equation}\label{eq:norm_bv}
\|h\|_{1,q}=\lim_{\ve\to 0}\inf\{\|g\|^-_{1,q} : g\in \cC^1(M,\bR)\text{ and }\|g-h\|_{0,q}\leq\ve\}.
\end{equation}
We then define $\cB^{1,q}:=\{h\in\cB^{0,q}\;|\; \|h\|_{1,q}<\infty\}$. One can see Section 2.7 of \cite{BKL} for a brief discussion of the general properties of such a construction. 

The next Lemma explains in which sense $\cB^{1,q}$ is a space of distributions.
\begin{lem}\label{lem:embedding}
The spaces $\cB^{i,q}$, $i\in\{0,1\}$, are spaces of distributions in the sense that there exist canonical embeddings $\iota_{0,q}:\cB^{0,q}\to (\cC^{q})'$ and $\iota_{1,q}:\cB^{1,q}\to (\cC^{q+1})'$.
\end{lem}
\begin{proof}
To start, note that there exist $C_q>0$ such that for all $\vf\in\cC^q$ and $W\in \cW^r_L$, $\|\vf\|^W_q\leq C_q\|\vf\|_{\cC^q}$. 
In addition, for each $h\in\cC^1(M,\bC)$,
 \[
\left|\int_Mh\; d\omega\right|\leq \|h\|_{0,q}\leq a^{-1}\|h\|_{1,q}
\]
which, by density, implies that  $\ell(h)=\int_Mh\; d\omega$ belongs to $(\cB^{0,q})'\subset(\cB^{1,q})'$, the duals of  $\cB^{0,q}$ and $\cB^{1,q}$, for each $q\geq 0$. Also, one can easily check that for each $\vf\in \cC^q$ and $h\in  \cB^{0,q}$, we have $\vf h\in  \cB^{0,q}$ and for each $\vf\in \cC^{q+1}$ and $h\in  \cB^{1,q}$, we have $\vf h\in  \cB^{1,q}$. This implies that  $\iota_{0,q}(h)(\vf):=\ell(\vf h)$ is well defined for each $\vf\in\cC^q$ and $h\in\cB^{0,q}$. In addition, for each $h\in\cC^1(M,\bC)$ we have
\[
|\iota_{0,q}(h)(\vf)|=\left|\int_M\vf h\; d\omega\right|\leq \|h\|_{0,q} \|\vf\|^W_q\leq C_q \|h\|_{0,q} \|\vf\|_{\cC^q},
\]
from which, by density, it follows $|\iota_{0,q}(h)(\vf)|\leq C_q \|h\|_{0,q} \|\vf\|_{\cC^q}$, for all $h\in \cB^{0,q}$ and $\vf\in\cC^q$.
That is $\iota_{0,q}(h)\in (\cC^q)'$. Thus $\iota_{0,q}:\cB^{0,q}\to (\cC^{q})'$, it remain to check that it is injective. Suppose that $\iota_{0,q}(h_0)=\iota_{0,q}(h_1)$, then for all $\vf\in \cC^{q}$ we have
\[
\int_M(h_0-h_1)\vf=0
\]
which, recalling \eqref{eq:dual}, implies $\|h_0-h_1\|_{0,q}=0$.\\
The other embedding is proven similarly.
\end{proof}
\begin{rem}\label{rem:identify} Form now on we will identify, when needed, the spaces $\cB^{i,q}$ with the spaces $\iota_{i,q}(\cB^{i,q})$ of distributions without further notice.\footnote{ In general  $\iota_{i,q}(\cB^{i,q})$ is not closed in the $(\cC^{q+i})'$ topology. To see  $\iota_{i,q}(\cB^{i,q})$ as a Banach space we have to induce the norm:  if $h\in \iota_{i,q}(\cB^{i,q})$, then $\|h\|=\|\iota_{i,q}^{-1}(h)\|_{i,q}$. Obviously, in this way $\iota_{i,q}$ becomes an isomorphism of Banach spaces, hence the possibility to identify them.}
\end{rem}
To better understand the  $\cB^{i,q}$ spaces it is useful to note that in special cases they are simply functions.

\begin{rem}\label{rem:no-stable} If $T$ is an expanding map, hence $d_s=0$, then the leaves are just points and $\|\vf\|^W_q=|\vf|_\infty$. The reader can easily check that $\cB^{0,q}=L^1$ and $\cB^{1,q}=BV$, as announced.
\end{rem}

\begin{rem}\label{rem:BV} By the definition \eqref{eq:norms} it follows that
\[
\begin{split}
&\sup_{(W,\vf)\in\Omega_{L,q, 1}}\left| \int_M \vf h \right|\leq \sup_{|\vf|_{\infty}\leq 1}\left| \int_M h\vf \right|= \|h\|_{L^1}\\
&\sup_{(W,\vf)\in\Omega_{L,q+1,d}}\left| \int_M h \, \div \vf \right |\leq \sup_{\|\vf\|_\infty\leq 1}\left| \int_M h \, \div \vf\right|\leq \|h\|_{BV}.
\end{split}
\]
Thus, by \eqref{eq:norms} and \eqref{eq:norm_bv},  $\|h\|_{1,q}\leq C_a\|h\|_{BV}$. That is $L^1\subset \cB^{0,q}$ and $BV\subset \cB^{1,q}$.
\end{rem}

\begin{rem}
There is no problem in considering norms with higher smoothness, as in \cite{GL}. We avoid it since it is not relevant for the issue we are presently exploring.
\end{rem}
\section{A Lasota-Yorke inequality}\label{sec:lasota}

Our first goal is to show that $\cL$ is bounded in the $\|\cdot\|_{0,q}, \|\cdot\|_{1,q}$ norms, hence $\cL$ extends uniquely to a bounded operator on $\cB^{0,q}$ and $\cB^{1,q}$.

To prove our basic proposition (a Lasota-Yorke type inequality) we need first a small approximation Lemma.
\begin{lem}\label{lem:app}
There exists $c_\param>1$, $\ve_0>0$ such that, for each  $q\in \{1,\dots, r\}$, $(W,\vf)\in \Omega_{L,q,1}$, and $\ve\in (0,\ve_0)$ there exists $\vf_\ve\in \cC^{r}(M, \bC)$ such that $(W,c_{ \param}^{-1}\vf_\ve)\in \Omega_{L,q,1}$, $(W,c_{\param}^{-1}\ve\vf_\ve)\in \Omega_{L,q+1,1}$ and $\|\vf-\vf_\ve\|^W_{q-1}\leq C_{q,\param}\ve$.
\end{lem}
\begin{proof}
Let $(W,\vf)\in \Omega_{L,q,1}$. Consider a mollifier $\jj_\ve(y)=\ve^{-d_s}\jj(\ve^{-1} y)$ where $\jj\in\cC^\infty$ is supported in a fixed ball.  Then, for $\ve\leq \delta_0/2$, and $(x,y)\in U^0$ define\footnote{See Remark \ref{rem:foliation} for the definition of  $\bF_{\xi}$.}
\begin{equation}\label{eq:molly}
\begin{split}
\widehat\vf_{i,\ve} (x,y)&=\int_{\bR^{d_s}}\vf\circ \phi_i^{-1}\circ( \bF_{\xi}(x, y+z)) \jj_\ve(z)dz\\
\vf_\ve& =\sum_i \vartheta_i \cdot \widehat\vf_{i,\ve}\circ \bF_{\xi}^{-1}\circ \phi_i. 
\end{split}
\end{equation}
Clearly $\vf_\ve\in\cC^r$, hence we only have to verify the other two properties. Note that
\[
\vf_\ve\circ \phi_{j}^{-1}\circ \bF_{\xi}(x,y)=\sum_i \left[\vartheta_i \cdot \widehat\vf_{i,\ve}\circ \bF_{\xi}^{-1}\circ \phi_i\right]\circ \phi_{j}^{-1}\circ \bF_{\xi}(x,y).
\]
By definition $\bF_{\xi}^{-1}\circ \phi_i\circ \phi_{j}^{-1}\circ \bF_{\xi}(x,y)=(h_{ij}(x),g_{ij}(x,y))$ for some $g_{ij}(x,\cdot)\in\cC^r$, moreover $\sup_{ij}\sup_x\|g_{ij}(x,\cdot)\|_{\cC^r}\leq C$ for some constant $C>0$. Thus
\[
\begin{split}
\vf_\ve\circ \phi_{j}^{-1}\circ \bF_{\xi}(x,y)&=\sum_i \vartheta_i\circ \phi_{j}^{-1}\circ \bF_{\xi}(x,y)\int_{\bR^{d_s}}\vf\circ \phi_i^{-1}\circ  \bF_{\xi}
((h_{ij}(x),g_{ij}(x,y)+z) \jj_\ve(z)dz.
\end{split}
\]
Using the formula above and \eqref{eq:molly} we can estimate
\[
\begin{split}
\|\vf_\ve\circ \phi_{j}^{-1}\circ \bF_{\xi}(x,\cdot)\|_{\cC^q}&\leq \sum_i \|\vartheta_i\circ \phi_{j}^{-1}\circ \bF_{\xi}(x,\cdot)\|_{\cC^r} \|\widehat\vf_{i,\ve}(h_{ij}(x),g_{ij}(x,\cdot)\|_{\cC^q}\leq c_{\param},
\end{split}
\]
for some constant $c_{\param}$.
On the other hand, recalling \eqref{def of the C norm},
\[
\begin{split}
&\|\vf_\ve\circ \phi_{j}^{-1}\circ \bF_{\xi}(x,\cdot)\|_{\cC^{q+1}}\leq C_\param \|\vf_\ve\circ \phi_{j}^{-1}\circ \bF_{\xi}(x,\cdot)\|_{\cC^{q}}\\
&+\sum_{il,|\alpha|=q}C \left\|\int_{\bR^{d_s}}\partial_{z_l}\partial^\alpha\left[\vf\circ \phi_i^{-1}\circ  \bF_{\xi}\right]((h_{ij}(x),z+g_{ij}(x,\cdot))\partial_{y_l} g_{ij}(x,\cdot) \prod_{k=1}^{q} \partial_{y_{\alpha_k}} g_{ij}(x,\cdot) \cdot \jj_\ve(z)dz\right\|_{\cC^0}\\
&\leq C_\param c_\param+ \sum_{il,|\alpha|=q}C \Bigg\|\int_{\bR^{d_s}}\partial^\alpha\left[\vf\circ \phi_i^{-1}\circ  \bF_{\xi}\right]((h_{ij}(x),z+g_{ij}(x,\cdot)) \partial_{y_l} g_{ij}(x,\cdot) \\
&\phantom{\leq \frac{c_\param}2 + \sum_{il,|\alpha|=q}C \Bigg\|}
\times\prod_{k=1}^{q} \partial_{y_{\alpha_k}} g_{ij}(x,\cdot) \partial_{z_j}\jj_\ve(z)dz\Bigg\|_{\cC^0}\\
&\leq C_\param c_{\param}+\|\vf\|^W_{q}C\Const \ve^{-1}\leq c_{\param}\ve^{-1},
\end{split}
\]
provided we chose $c_\param>2C\Const$ and $\ve_0<(2C_\param)^{-1}$.
To verify the last inequality note that there exists a constant $C_{q, \param}>0$ such that
\[
\left\|\vf\circ \phi_i^{-1}\circ  \bF_{\xi}((h_{ij}(x),z+g_{ij}(x,\cdot))-\vf\circ \phi_i^{-1}\circ  \bF_{\xi}((h_{ij}(x),g_{ij}(x,\cdot))\right\|_{\cC^{q-1}} \leq C_{q, \param} \|\vf\|^W_q |z|.
\]
Hence,
\[
\|\vf-\vf_\ve\|^W_{q-1}\leq C_{q, \param}\ve.
\]
\end{proof}

\begin{prop}\label{lem:norm}
For each $\sri\in (\max\{\nu,\lambda^{-1}\},1)$, we can chose $\varpi>2$ such that there exist constants $a, A, B>0$  such that, for all $h\in\cC^1(M, \bC)$, $q\in \{0,\dots,r- 1\}$, holds true
\[
\|\cL^n h\|_{0,q}\leq A \|h\|_{0,q}.
\]
 In addition, for all $q\in \{1,\dots,r- 2\}$, holds true
\[
\begin{split}
&\|\cL^n h\|_{0,q}\leq A \sri^n\|h\|_{0,q}+B\|h\|_{0,q+1};\\
&\|\cL^n h\|_{1,q}\leq A\sri^n\|h\|_{1,q}+B \|h\|_{0,q+1}.
\end{split}
\]
\end{prop}
\begin{proof}
Note that if  $(W,\vf)\in \Omega_{L,q,1}$, then, by \eqref{def of the C norm} and \eqref{eq:anosov} and for $A$ large enough,
\begin{equation}\label{eq:zeroq}
\left |\int_M \cL^n h \vf \right|=\left |\int_M h \vf\circ T^n\right|\leq A \|h\|_{0,q}\,,
\end{equation}
from which the first inequality follows.

For each $\ve>0$  and $(W,\vf)\in \Omega_{L,q,1}$ we define $\vf_\ve$ as in Lemma \ref{lem:app}. Hence,
\[
\begin{split}
\left |\int_M \cL^n h \vf \right|&=\left |\int_M h (\vf-\vf_\ve)\circ T^n\right|+\left |\int_M h \vf_\ve\circ T^n\right|\\
&\leq \|(\vf-\vf_\ve)\circ T^n\|^{T^{-n}W}_{q}\|h\|_{0,q}+ \|\vf_\ve\circ T^n\|^{T^{-n}W}_{q+1}\|h\|_{0,q+1}\,.
\end{split}
\]
Then, by Lemmata \ref{lem:dual} and \ref{lem:app},
\[
\begin{split}
\left| \int_M \cL^n h \vf\right|&\leq \|(\vf-\vf_\ve)\circ T^n\|^{T^{-n}W}_{q}\|h\|_{0,q} +\|\vf_\ve\circ T^n\|^{T^{-n}W}_{q+1}\|h\|_{0,q+1}\\
&\leq \left(A_0\sigma^{qn}\|\vf-\vf_\ve\|^W_{q}+ B_0\|\vf-\vf_\ve\|^W_{q-1}\right)\|h\|_{0,q}+
 A_0\|\vf_\ve\|^{W}_{q+1}\|h\|_{0,q+1} \\
&\leq \left(2c_{\param}^{-1} A_0 \sigma^{qn}+C_{q,\param}B_0\ve\right)\|h\|_{0,q}+
A_0c_{\param} \ve^{-1}\|h\|_{0,q+1}.
\end{split}
\]
For each $\sri\in (\sigma,1)$ there exists $n_1\in\bN$ and $\ve$ such that\footnote{Recall that for this statement we require $q\ge1$. It is obvious that $q=0$ does not lead to any contraction. This point shows the need to work with a space of distributions rather than a space of measures.}   $2A_0 \sigma^{q n_1}+C_{q,\param}B_0\ve\leq
{\sri^{{2}n_1}}$. Thus, taking the sup for $(W,\vf)\in\Omega_{L,q,1}$ we have, for $n\in\{ n_1,\dots, 2n_1\}$,
\[
\begin{split}
&\|\cL^{n} h\|_{0,q}\leq \sri^{n}\| h\|_{0,q}+ C_{\param} \|h\|_{0,q+1}.
\end{split}
\]
Iterating yields that there exists $A_1>0$ such that,
\begin{equation}\label{eq:dep_weak}
\begin{split}
&\|\cL^n h\|_{0,q}\leq \sri^n\|h\|_{0,q}+B_\param\|h\|_{0,q+1} {\quad \textrm{ for all } n\geq n_1}\\
& \|\cL^n h\|_{0,q}\leq A_1 \sri^n\|h\|_{0,q}+B_\param\|h\|_{0,q+1} \quad \textrm{ for all } n\in\bN.
\end{split}
\end{equation}
Next, we prove the third inequality in the statement of the lemma. For each $(W,\vf)\in\Omega_{L,q+1,d}$ write
\[
\int_M \cL^n h \, \div \vf=\int_M h (\div \vf)\circ T^n.
\]
Note that, setting $R=\phi_i\circ T^n\circ \phi_j^{-1}$ and recalling footnote \ref{foo:dive}, we have
\begin{equation}\label{eq:unstable}
\begin{split}
\div (DR^{-1}\circ \phi_j \cdot \vf\circ T^n)=(\div \vf)\circ T^n+\sum_{l,k=1}^d \partial_{l} \left[(DR)^{-1}\right]_{lk}\vf_{k}\circ T^n\circ \phi_i^{-1}.
\end{split}
\end{equation}
Set  $D_n=\sup_l\sum_{k}\| \left[\partial_l (DR^{-1})_{l,k}\right]\|_{\cC^r}$.

It is then natural to decompose $\vf$ into an ``unstable" and a ``stable" part. More precisely consider the ``almost unstable" foliation $\Gamma=\{\gamma_s\}_{s\in\bR^{d_s}}$ made of the leaves, in some chart $\phi_j$, $\gamma_s=\{(u,s)\}_{u\in\bR^{d_u}}$ and its image $T^n \Gamma$. The leaves of $T^n \Gamma$ can be expressed, in some chart $\phi_i$, in the form $\{(x, \tilde G_n(x,y)\}$ for some function $\tilde G_n$, smooth in the $x$ variable,  with $\|\partial _x \tilde G_n\|\leq 1$ and the normalization $\tilde G_n(F(0,y),y)=y$. On the other hand the leaves of $W$, in the same chart, have the form $\{(F(x,y),y)\}$. It is then natural to consider the change of variables $(x,y)=\Psi_n(x',y')$ where $(x,\tilde G_n(x,y'))=(F(x',y), y)$. Writing $\vf=(\vf_1,\vf_2)$, with $\vf_1\in \bR^{d_u}$, $\vf_2\in\bR^{d_s}$ we consider the decomposition\footnote{Since $v$ and $w$ depend on $n$, a more precise notation would be $v_n, w_n$.  We suppress the subscript $n$ to ease notation and since no ambiguity can arise.}
\begin{equation}\label{eq:vfvw}
\begin{split}
\vf\circ \phi_i^{-1}&(x,y)=\vf^u \circ \phi_i^{-1}(x,y)+\vf^s\circ \phi_i^{-1}(x,y)\\
&=(v(x,y), \partial_x \tilde G_n(x,y') v(x,y))+(\partial_y F(x',y) w(x,y), w(x,y)).
\end{split}
\end{equation}
That is, setting $\hat\vf= \vf\circ \phi_i^{-1}$,
\begin{equation}\label{eq:vw}
\begin{split}
&v(x,y)=(\Id-\partial_y F(x',y) \partial_x\tilde G_n(x,y'))^{-1}(\hat \vf_1(x,y)-\partial_y F(x',y) \hat \vf_2(x,y))\\
&w(x,y)= (\Id-\partial_x \tilde G_n(x,y')\partial_y F(x',y))^{-1}(\hat \vf_2(x,y)-\partial_x \tilde G_n(x,y') \hat \vf_1(x,y)).
\end{split}
\end{equation}
Thus, recalling equation \eqref{eq:unstable} and Lemma \ref{lem:dual},
\begin{equation} \label{eq:partone}
\begin{split}
\left |\int_M\cL^nh \,\div \vf \right|\leq& \left |\int_M\cL^nh \,\div \vf^s\right|+\left| \int_M\cL^nh \,\div \vf^u\right|\\
\le& \left |\int_M\cL^nh \,\div \vf^s\right|+ A_0 D_n\|h\|_{0,q+1}\\
&+\left|\int_Mh \,\div( \left[(DR)^{-1}\circ R^{-1}\circ  \phi_i\cdot \vf^u\right]\circ T^n)\right|.
\end{split}
\end{equation}
To estimate the above terms our first task is to compute the norm of $\div(\vf^s)$, $(x,y)=\bF(x',y):=(F(x',y),y)$. We start noticing that
\[
\begin{split}
\sum_i\partial_{y_i} \left[w_i \circ \bF\right](x',y)=&\sum_{i,j} (\partial_{x_j} w_i)(x,y)\cdot\partial_{y_i}F_j(x',y)+\sum_i(\partial_{y_i}w)(x,y)\\
=&\div (\vf^s)(x,y)-\sum_{i,j} \left[\left(\partial_{x'_k}\partial_{y_i}F_j\cdot(\partial_x F)_{kj}^{-1}\right)\circ \bF^{-1}\cdot w_i\right](x,y) .
\end{split}
\]
Accordingly, recalling \eqref{eq:HF},
\begin{equation}\label{eq:H_use}
\begin{split}
&\div (\vf^s)(F(x',y),y))=\left[\sum_i\partial_{y_i} [w_i \circ \bF]+\sum_{i} \left(H^F_i\cdot w_i\right)\circ \bF\right] (x',y)\\
&w\circ\bF(x',y)=(\Id-\partial_x \tilde G_n(F(x',y),y)\partial_y F(x',y))^{-1}\left[\vf_2\circ\bF-\partial_x \tilde G_n \vf_1\circ\bF)\right](x',y).
\end{split}
\end{equation}
Since, $\|w\|^W_{q+1}\leq C_{n,\param}$, recalling Definition \ref{def:foliation} for all $|\alpha|\leq q$ we have 
\[
|\partial^\alpha_y[\div (\vf^s)\circ \phi_i^{-1}\circ \bF](x',\cdot)|\leq C_{n}.
\]
Hence, by  \eqref{eq:zeroq},
\begin{equation}\label{eq:first-div}
 \left |\int_M\cL^nh \,\div \vf^s\right|\leq C_n \|\cL^n h\|_{0,q}\leq C_nA\|h\|_{0,q}.
\end{equation}
On the other hand, for each $|\alpha|=q+1$, using \eqref{eq:vfvw} and \eqref{eq:vw} we have

\[
\begin{split}
&\left|\partial^\alpha_y\left\{\left[(DR)^{-1}\circ R^{-1}\vf^u\circ \phi_i^{-1}\right]\circ \bF \right\}(x',y)\right|\\
&\leq \left|(DR\circ R^{-1}\circ \bF (x',y))^{-1}\begin{pmatrix}\Id&0\\0&\partial_x\tilde G(\bF(x',y))\end{pmatrix}\partial^\alpha_y\left[v\circ \phi_i^{-1}\circ \bF\right] (x',y)\right|+C_n\param^{-1}\|\vf\|^W_q
\end{split}
\]
where the last term bounds all the terms with at most $q$ derivatives on $v$. Since the range of the matrix in the line above belongs to the image of the unstable cone under $R$, by \eqref{eq:anosov} (and putting in the remainder all the terms with at most $q$ derivatives of $\vf$) we have
\[
\begin{split}
&\left|\partial^\alpha_y\left\{\left[(DR)^{-1}\circ R^{-1}\vf^u\circ \phi_i^{-1}\right]\circ \bF \right\}(x',y)\right|\leq \frac{1+\sri}{\czero}\lambda^{-n} \left|\partial^\alpha_y\left[v\circ \phi_i^{-1}\circ \bF\right] (x',y)\right|\\
&+C_n\param^{-1}\|\vf\|^W_q\leq \frac{(1+\sri^2)^2}{\czero(1-\sri)}\lambda^{-n}\|\vf\|^W_{q+1}+C_n\param^{-1}\|\vf\|^W_q.
\end{split}
\]
Accordingly
\[
\|(DR)^{-1}\circ R^{-1}\circ  \phi_i\cdot \vf^u\|^W_{q+1}\leq \frac{(1+\sri^2)^2}{\czero(1-\sri)}\lambda^{-n}\|\vf\|^W_{q+1}+C_n\param^{-1}\|\vf\|^W_q.
\]
Then Lemma \ref{lem:dual}  implies
\begin{equation}\label{eq:vfu}
\|\left[(DR)^{-1}\circ R^{-1}\circ  \phi_i\cdot \vf^u\right]\circ T^n\|^{T^{-n}W}_{q+1}\leq \frac{A_0(1+\sri^2)^2}{\czero(1-\sri)\lambda^{n}}\|\vf\|^W_{q+1}+\frac{C_n}{\param}\|\vf\|^W_q.
\end{equation}

We can now chose $n_2\in\bN$,  $n_2\geq n_1$, such that 
\[
\frac{A_0(1+\sri)^2}{\czero(1-\sri)}\lambda^{-n_2}\leq \frac 14\sri^{ n_2}
\]
and finally we choose $\param$ such that
\begin{equation}\label{eq:param-sel}
4\sup_{l\leq 2n_2}C_{l}\sri^{-2n_2}\leq \param.
\end{equation} 
Accordingly,  for all $n\in \{n_2,\dots, 2n_2\}$,
\[
\|\left[(DR)^{-1}\circ R^{-1}\circ  \phi_i\cdot \vf^u\right]\circ T^n\|_{q+1}^{T^{-n}W}\leq{ \frac12}\sri^{n}\|\vf\|_{q+1}^W,
\]
We can then continue the estimate started in \eqref{eq:partone}, recalling \eqref{eq:first-div} we have:\footnote{Note that in  \eqref{eq:param-sel} we have chosen $\param$ and that the choice depends only on $T$, thus we can drop the $\param$ dependency from all the constants.}
\[
\left |\int_M\cL^nh \,\div \vf \right|\le  AC_n\|h\|_{0,q}+\frac 12\sri^{n}\|h\|_{1,q}^*+ A_0D_{n_2}\|h\|_{0,q+1}.
\]
Finally, choose $a$ such that $\sup_{l\leq 2n_2} C_{l}A a^{-1}\leq \frac 12\theta^{2n_2}$, then taking the sup on $\vf, W$ we have, for all $n\in\{n_2,\dots, 2n_2\}$, and using \eqref{eq:dep_weak},
\[
\|\cL^{n}h\|^-_{1,q}\le {\sri^{n}}\|h \|^-_{1,q}+ B_{n_2}\|h\|_{0,q+1}.
\]
Then, for each $n\in\bN$  we can write $n=kn_2+m$, $m\leq n_2$ and iterating the above inequality we have, for all $n\in\bN$,
\[
\|\cL^{n}h\|^-_{1,q}\le A{\sri^{n}}\|h \|^-_{1,q}+ B\|h\|_{0,q+1}.
\]
Finally, if $h\in\cB^{1,q}$, then there exists $\{g_k\}\in \cC^1$: $g_k\stackrel{\cB^{0,q}}{\to} h$ and $\|g_k\|^-_{1,q}\to\|h\|_{1,q}$. Since, $\cL^ng_k\in \cC^1$ and $\cL^n g_k\to \cL^n h$ in $\cB^{0, q}$ we have
\[
\begin{split}
\|\cL^nh\|_{1,q}&\leq \lim_{k\to \infty} \|\cL^n g_k\|^-_{1,q}\leq A{\sri^{n}}\lim_{k\to \infty}\|g_k \|^-_{1,q}+ B\lim_{k\to \infty}\|g_k\|_{0,q+1}\\
&=A{\sri^{n}}\|h \|_{1,q}+ B\|h\|_{0,q+1}.
\end{split}
\]
This finishes the proof of the second item in the proposition. The proof of the first item of the proposition follows from \eqref{eq:dep_weak} and \eqref{eq:param-sel}.
\end{proof}
\section{On the essential spectrum}\label{sec:comp}
In the previous section we have seen that $\cL$ (or rather its extension that, with a slight abuse of notation, we still call $\cL$) belongs both to $L(\cB^{0,q},\cB^{0,q})$ and $L(\cB^{1,q},\cB^{1,q}$). Moreover Proposition \ref{lem:norm} implies that the spectrum of $\cL$ is contained in the unit disc. Next we want to study the essential spectrum (that is the complement of the point spectrum with finite multiplicity).

\begin{lem}\label{lem:ess} For $q\in \{1,\dots,r-2\}$, the essential spectrum of $\cL$ on $\cB^{1,q}$ is contained in the disc $\{z\in\bC\;:\; |z|\leq \max\{\lambda^{-1},\nu\}\}$.
\end{lem}
\begin{proof}
By Lemmata \ref{lem:foliation} and \ref{lem:dual} it follows that it suffices to study the sup of $\int_M h\vf$ for $(W,\vf)\in\Omega_{L/4,q+1,1}$. Indeed, if $B^-_1=\{h\in\cB^{1,q}\;:\;\|h\|^-_{1,q}\leq 1\}$ is relatively compact in the topology associated to the norm $\|h\|_{0,q+1}'= \sup_{(W,\vf)\in\Omega_{L/4,q+1,1}}\left|\int_M h\vf\right|$, then, by Lemmata \ref{lem:foliation} and \ref{lem:dual}, there exists $n_0\in\bN$ such that $\|\cL^{n_0}h\|_{0,q+1}\leq \|h\|_{0,q+1}'$. Hence, $\cL^{n_0}B^-_1$ is relatively compact in $\cB^{1,q}$, thus $\cL^{n_0}$ is compact as an operator from $\cB^{1,q}$ to  $\cB^{0,q+1}$ and the Lemma follows from Proposition \ref{lem:norm} and the usual Hennion argument \cite{He} based on Nussbaum essential spectral formula \cite{Nu}, see \cite{L} for details.

Let us prove the relative compactness of $B_1^-$. Since we can write
\[
\int_M h\vf =\sum_i\int_M h \vartheta_i\vf
\]
we can assume, without loss of generality, that $\vf$ is supported in a given chart $(V_i,\phi_i)$. From now on we will work in such a chart without further mention.

Let us define $\vf_t$ to be the solution of the heat equation
\[
\begin{split}
\partial_t \vf_t&=\Delta_x\vf_t \quad \text{in }\bR^d\times[0,1]\\
\vf_0&=\vf.
\end{split}
\]
That is
\begin{equation}\label{eq:heat}
\vf_t(x,y)=\frac1{(4\pi t)^{d_u/2}}\int_{\bR^{d_u}}e^{-\frac{|\zeta|^2}{4t}}\vf(x-\zeta,y)d\zeta.
\end{equation}
Then, for each small $\ve>0$,
\[
\begin{split}
\int_{M} h\vf&=\int_M h\vf_\ve-\int_0^\ve dt \int_M h\partial_t\vf_t= \int_{V_i} h\,\vf_\ve-\int_0^\ve dt \int_{V_i} h\,\div \nabla_x\vf_t(x,y)\\
& =\int_{V_i} h\,\vf_\ve+\int_0^\ve dt \frac1{(4\pi t)^{d_u/2}}\int_{V_i}dx dy\int_{\bR^{d_u}}d\zeta e^{-\frac{|\zeta|^2}{4t}} h\, \div \nabla_\zeta \vf(x-\zeta,y)\\
&=\int_{V_i} h\,\vf_\ve+\int_0^\ve dt \frac1{(4\pi t)^{d_u/2}}\int_{\bR^{d_u}}d\zeta e^{-\frac{|\zeta|^2}{4t}}\int_{V_i} h\,\div \frac{\zeta}{2t}\vf^\zeta,
\end{split}
\]
where, in the last line,  $\vf^\zeta(x,y):=\vf(x-\zeta,y)$ and we have integrated by part with respect to $\zeta$.
Next, for each $\zeta\in\bR^{d_u}$ we define the foliation $\bF_\zeta(x,y):=(F(x-\zeta,y)+\zeta,y)$, note that the foliation $W^\zeta$ defined by $\bF_\zeta$ belongs to $\cW^r_{L/4}$. Then $\vf^\zeta\circ\bF_\zeta(x,y)=\vf(F(x-\zeta,y),y)$ which implies $\|\vf^\zeta\|_{q+1}^{W_\zeta}\leq 1$. Hence,
\begin{equation}\label{eq:compact1}
\int_{M} h\vf=\int_{V_i} h\,\vf_\ve+\cO(\|h\|^-_{1,q}\ve).
\end{equation}
In addition, by \eqref{eq:heat} and integrating $r$ times by parts
\[
|\vf_\ve(\cdot, y)|_{\cC^r}\leq \frac {C_r}{(4\pi \ve)^{d_u/2}}\int_{\bR^{d_u}}e^{-\frac{|\zeta|^2}{4\ve}}(\ve^{-\frac r2}+\ve^{-r}\|\xi\|^r)\|\vf\|_{\cC^0}d\zeta\leq C\ve^{-\frac r2}.
\] 
Moreover, recalling \eqref{eq:heat}, Definition \ref{def:foliation} and Lemma \ref{lem:holo},
\[
\begin{split}
\vf_\ve\circ \bF(x,y)&=\frac1{(4\pi \ve)^{d/2}}\int_{\bR^{d_u}}e^{-\frac{|F(x,y)-\zeta|^2}{4\ve}}\vf(\zeta, y)d\zeta\\
&=\frac1{(4\pi \ve)^{d/2}}\int_{\bR^{d_u}}e^{-\frac{|F(x,y)-F(\xi,y)|^2}{4\ve}}\vf(F(\xi,y), y)\det(\partial_xF)(\xi,y) d\xi\\
\end{split}
\]
which readily implies $\|\vf_\ve\|^W_r\leq C \ve^{-r}$.
This, by \cite{J}, implies that $|\vf_\ve|_{\cC^r}\leq C\ve^{-r}$.
Thus, recalling \eqref{eq:compact1}, we have, for each $\ve>0$,
\begin{equation}\label{eq:compact2}
\|h\|_{0,q+1}\leq C\ve^{-r}\|h\|_{(\cC^r)'}+C\|h\|^-_{1,q}\ve.
\end{equation}
Since $(\cC^{q+1})'$ embeds compactly in $(\cC^r)'$ and Lemma \ref{lem:embedding} implies that $B_1^-$ is a bounded subset of $(\cC^{q+1})'$ it follows that $B_1^-$ is relatively compact in $(\cC^r)'$. From this and equation \eqref{eq:compact2} the relative compactness of $B_1^-$ in $\cB^{0,q+1}$  readily follows. Hence the Lemma.
\end{proof}

\section{On the peripheral spectrum}\label{sec:top}
The previous section implies, for each $ \beta\in (\max\{\lambda^{-1},\nu\}, 1)$, the spectral decomposition
\begin{equation}\label{eq:dec}
\cL=\sum_{j=1}^{L_{\beta}}\lambda_j\Pi_j+R
\end{equation}
where $\Pi_j\Pi_k=\delta_{jk}\Pi_k^2$, $\Pi_j R=R\Pi_j=0$, each $\Pi_j$ is a finite rank operator, and the spectral radius of $R$ is bounded by $\beta$. 
\begin{lem}\label{lem:peripheral}
One is an eigenvalue of $\cL$.  Letting  $h_*:=\Pi_11$,  $h_*$ is a measure. In addition, the peripheral spectrum of $\cL$ consists of finitely many finite groups.
 \end{lem}
 \begin{proof}
Since $\ell$ is an eigenvalue of the dual of $\cL$ (the Lebesgue measure being the eigenvector), it must belong to the spectrum of $\cL$. Next, we choose $\beta$, in the representation \eqref{eq:dec}, large enough so that for all the eigenvectors we have $|\lambda_j|=1$. In this case, since the operator is power bounded, the $\Pi_j$ cannot contain Jordan blocks, thus $\Pi_j\Pi_k=\delta_{jk}\Pi_k$.
A simple computation based on \eqref{eq:dec} shows
\begin{equation}\label{eq:peri-sp}
\lim_{n\to\infty}\frac 1n\sum_{k=0}^{n-1}e^{-i\vartheta}\cL^k=\begin{cases}0&\text{if } e^{i\vartheta}\notin\sigma(\cL)\\
                                                                                           \Pi_j&\text{if }e^{i\vartheta}=\lambda_j =:e^{i\vartheta_j}.
                                                                                           \end{cases}
\end{equation}
For each $\vf\in\cC^q$ holds
\[
\left|\int_Mh_* \vf\right|\leq\lim_{n\to\infty}\frac 1n\sum_{k=0}^{n-1}\int_M \cL^k1 |\vf|= \lim_{n\to\infty}\frac 1n\sum_{k=0}^{n-1}\int_M  |\vf\circ T^k|\leq |\vf|_\infty.
\]
In other words $h_*$ defines measure.
Then let $h\in\cC^1$ and  $\vf\in\cC^q$, $\vf \geq 0$,
\[
\begin{split}
\left|\int_M\Pi_j h \vf\right|&\leq\lim_{n\to\infty}\frac 1n\sum_{k=0}^{n-1}\int_M \cL^k|h| \vf= \lim_{n\to\infty}\frac 1n\sum_{k=0}^{n-1}\int_M |h| \vf\circ T^k\\
&\leq|h|_\infty\lim_{n\to\infty}\frac 1n\sum_{k=0}^{n-1}\int_M \cL^k 1 \vf=|h|_\infty \int_M h_* \vf.
\end{split}
\]
Moreover, by a similar computation,
\[
\begin{split}
\left|\int_M\Pi_j h \vf\right|&\leq\lim_{n\to\infty}\frac 1n\sum_{k=0}^{n-1}\int_M \cL^k|h| \vf\leq |\vf|_\infty\int_M |h| .
\end{split}
\]
This implies $\Pi_j h=\sum_{l=1}^{n_j} \psi_{j,l}h_*\int_M h\phi_{j,l} $ where $\psi_{j,l},\phi_{j,l}\in L^\infty(M)$.
Note that, $\Pi_k\Pi_m=\delta_{km}\Pi_k$ implies 
\begin{equation}\label{eq:orto}
\int_M \phi_{k,l}\psi_{m,l'}h_*=\delta_{k,m}\delta_{l,l'}.
\end{equation}
Accordingly, for all $g,h\in\cC^r$,
\[
\begin{split}
\sum_{l=1}^{n_j} \int_M g\psi_{j,l}h_*\int_M \phi_{j,l}\circ T  h&= \int_Mg  \Pi_j \cL h=e^{i\vartheta_j} \int_Mg  \Pi_j h\\
&=e^{i\vartheta_j}\sum_{l=1}^{n_j} \int_M g\psi_{j,l}h_*\int_M h\phi_{j,l}.
\end{split}
\]
It follows that $\phi_{j,l}\circ T= e^{i\vartheta_j} \phi_{j,l}$, $\omega$ almost surely. On the other hand
\[
\begin{split}
\sum_{l=1}^{n_j} \int_M g\cL\psi_{j,l}h_*\int_M \phi_{j,l}  h&= \int_Mg \cL \Pi_j  h=e^{i\vartheta_j} \int_Mg  \Pi_j h\\
&=e^{i\vartheta_j}\sum_{l=1}^{n_j} \int_M g\psi_{j,l}h_*\int_M h\phi_{j,l}.
\end{split}
\]
By the arbitrariness of $g,h$ it follows
\[
e^{i\vartheta_j}\psi_{j,l}h_*=\cL\psi_{j,l}h_*=\psi_{j,l}\circ T^{-1}\cL h_*=\psi_{j,l}\circ T^{-1}h_*,
\]
which implies $\psi_{j,l}\circ T^{-1}=e^{i\vartheta_j}\psi_{j,l}$, $h_*d \omega$ almost surely. Note that this implies that, for all $k\in\bN$, $\psi_{j,l}^k\circ T^{-1}=e^{i\vartheta_j k}\psi_{j,l}^k$, thus $\cL(\psi_{j,l}^kh_*)=\psi_{j,l}^k\circ T^{-1}\cL h_*=e^{i\vartheta_j k} h_*$. By an approximation argument one can prove that $\psi_{j,l}^kh_*\in \cB^{1,q}$. But then it follows that $\{e^{i\vartheta_j k}\}\subset \sigma(\cL)$ and since the operator is quasi-compact it can have only finitely many isolated eigenvalues. Thus, we must have $\vartheta_j=\frac{2\pi k_j}{n_j}$, which concludes the proof.
\end{proof}

Lemma \ref{lem:peripheral} implies that there exists $\bar m\in\bN$ such that the peripheral spectrum of $\cL^{\bar m}$ consists of only the eigenvalue 1 with associated eigenprojector  $\overline \Pi =\sum_{l=1}^{N} \psi_{l}h_*\int_M h\phi_{l} $ where the $\psi_{l}\in\{\psi_{j,i}\}$ and $\phi_{l}\in\{\phi_{j,i}\}$. Moreover, \eqref{eq:orto} implies
\begin{equation}\label{eq:orto2}
\int_M \phi_{l}\psi_{l'}h_*=\delta_{l,l'}.
\end{equation}

Accordingly, the rest of the spectrum will be contained in a disk strictly smaller than one: that is $\cL^{\bar m}=\overline \Pi+Q$ where $\|Q^n\|_{1,q}\leq C\sigma^n$ for some $C>0$ and $\sigma\in (0,1)$. In addition, note that \eqref{eq:peri-sp} implies $\overline \Pi 1=h_*$.

A more precise result can be easily obtained.
\begin{lem}\label{lem:spectra}
The ergodic decomposition of $h_*$ corresponds to the spectral decomposition for the Anosov map and consists of the physical measures.
\end{lem}
\begin{proof}
 Let $\overline \Pi'h=\sum_{l=1}^{N} \phi_{l}  \int_M\psi_{l}h_* h$ and recall that $\phi_{l}\circ T=\phi_{l}$. For each $h\in\cC^q(M,\bR)$ let $\hat h=h-\overline \Pi'h$. Then
\[
\begin{split}
&\int_M\left|\frac 1{n}\sum_{k=0}^{n-1}h\circ T^{\bar m k}-\overline \Pi' h\right|^2=\int_M\left|\frac 1{n}\sum_{k=0}^{n-1}\hat h\circ T^{\bar mk}\right|^2\\
&=\sum_{k,j=0}^{n-1}\frac 1{\bar n^2} \int_M \hat h\circ T^{\bar m k} \hat h\circ T^{\bar mj}=\sum_{k=0}^{n-1}\frac 1{n^2} \int_M\cL^{\bar m k} 1 \hat h^2+2\sum_{k>j=0}^{n-1}\frac 1{n^2} \int_M \hat h\cL^{\bar m k-\bar m j}\hat h\cL^{\bar m j} 1\\
&=\cO\left(\frac 1n\right)+2\sum_{j=0}^{n-1}\sum_{l=1}^{n-j-1}\frac 1{n^2} \int_M \hat h\cL^{\bar m l}\hat h\cL^{\bar m j} 1\\
&=\cO\left(\frac 1n\right)+2\sum_{j=0}^{n-1}\sum_{l=1}^{n-j-1}\frac 1{n^2}\int_M \hat h\overline\Pi\hat h\overline \Pi 1 +C\sum_{j=0}^{n-1}\sum_{l=1}^{n-j-1}\frac 1{n^2}(\sigma^{l}+\sigma^j)\\
&=\cO\left(\frac 1n\right)+2\sum_{j=0}^{n-1}\sum_{l=1}^{n-j-1}\frac 1{n^2}\int_M \hat h\overline\Pi\hat h\overline \Pi 1 . 
\end{split}
\]
Next note that, recalling \eqref{eq:orto2},
\[
\begin{split}
\int_M \hat h\overline\Pi\hat h\overline \Pi 1&=\int_M \hat h\overline\Pi\hat h h_*\\
&=\sum_{l=1}^{N} \int_M h\psi_{l}h_*\int_M \phi_{l}\hat h h_*-
\sum_{l=1}^{N}\sum_{j=1}^{N}  \int_M\psi_{l}h_* h\int_M  \phi_{l} \psi_{j}h_*\int_M \phi_{j}\hat h h_*=0.
\end{split}
\]
It follows
\begin{equation}\label{eq:L2}
\int_M\left|\frac 1n\sum_{k=0}^{n-1}h\circ T^{\bar mk}-\overline \Pi h\right|^2 \leq C_hn^{-1}.
\end{equation}
By Chebyshev this implies that 
\[
\omega\left(\left\{x\in M\;:\; \left|\frac 1n\sum_{k=0}^{n-1}h\circ T^{\bar mk}-\overline \Pi  h\right|\geq \ve\right\}\right)\leq \frac{C_{ h}}{\ve^2 n}.
\]
thus, if we consider $\alpha\in (0,1)$ and the set $I=\cup_{k\in\bN}\{2^k+j2^{\alpha k}\}_{0\leq j<2^{(1-\alpha)k}}$, each sequence $\{n_j\}\subset I$ will have limit $\omega$ a.s. by a standard Borel-Cantelli argument. On the other hand, since $h$ is bounded, this readily implies 
\[
\lim_{n\to\infty}\frac 1n\sum_{k=0}^{n-1}h\circ T^{\bar m k}=\overline \Pi 'h=\sum_{l=1}^{N} \phi_{l}  \int_M\psi_{l}h_* h\quad\text {$\omega$-a.s.}.
\]
By an obvious approximation argument the same can be proven for each $h\in\cC^0(M,\bR)$.
This implies that the ergodic decomposition of $h_*$ consists of the physical measures. It is well known that these are the SRB measures of the system.
\end{proof}
\begin{rem} If the map is topologically transitive, then the physical measure is unique and so are the physical measures of the powers of the map. Hence the map is mixing, and no other eigenvalue of modulus one exist. Thus, the transfer operator has a spectral gap and the map is exponentially mixing for $BV$ observables.
\end{rem}
\appendix
\section{Norms estimates}\label{sec:normscq}
We provide a few tools on how to estimate $\cC^q$ norms of products and compositions of functions. These are well known facts, yet it is not so easy to find in the literature the exact statements needed here, so we provide them for the reader's convenience.

\begin{proof}[{\bfseries Proof of Lemma \ref{properties of the C norm}}] 
Let $\vf,\psi\in\cC^\rho(M, \bC)$. First we prove, by induction on $\rho$,
\begin{equation}\label{eq:prod}
\sup_{|\alpha|=\rho}\|\partial^\alpha(\vf \psi)\|_{\cC^0}\leq \sum_{k=0}^\rho \binom{\rho}{k}\sup_{|\beta|=\rho-k}\|\partial^\beta\vf\|_{\cC^0}
\sup_{|\gamma|=k}\|\partial^\gamma\psi\|_{\cC^0}.
\end{equation}
Indeed, it is trivial for $\rho=0$ and
\[
\begin{split}
\|\partial_{x_i}\partial^\alpha(\vf \psi)\|_{\cC^0}&=\|\partial^\alpha(\psi\partial_{x_i}\vf+\vf\partial_{x_i}\psi)\|_{\cC^0}\\
&\leq \sum_{k=0}^\rho \binom{\rho}{k}\sup_{|\beta|=\rho-k}\|\partial^\beta\partial_{x_i}\vf\|_{\cC^0}
\sup_{|\gamma|=k}\|\partial^\gamma\psi\|_{\cC^0}\\
&\phantom{\leq}
+ \sum_{k=0}^\rho \binom{\rho}{k}\sup_{|\beta|=\rho-k}\|\partial^\beta\partial_{x_i}\psi\|_{\cC^0}
\sup_{|\gamma|=k}\|\partial^\gamma\vf\|_{\cC^0}\\
&\leq \sum_{k=0}^{\rho} \binom{\rho}{k}\sup_{|\beta|=\rho-k+1}\|\partial^\beta\vf\|_{\cC^0}
\sup_{|\gamma|=k}\|\partial^\gamma\psi\|_{\cC^0}\\
&\phantom{\leq}
+ \sum_{k=1}^{\rho+1} \binom{\rho}{\rho+1-k}\sup_{|\gamma|=k}\|\partial^\gamma\psi\|_{\cC^0}
\sup_{|\beta|=\rho-k+1}\|\partial^\beta\vf\|_{\cC^0}
\end{split}
\]
from which \eqref{eq:prod} follows taking the sup on $\alpha$, $i$ and since $\binom{\rho}{k}+\binom{\rho}{\rho+1-k}=\binom{\rho+1}{k}$.
The first statement of the Lemma readily follows:

\[
\begin{split}
\|\vf\psi\|_{\cC^\rho}&=\sum_{k=0}^{\rho}\param^{\rho-k}  \sum_{j=0}^k \binom{k}{j}\sup_{|\beta|=k-j}\|\partial^\beta\vf\|_{\cC^0}
\sup_{|\gamma|=j}\|\partial^\gamma\psi\|_{\cC^0}\\
&=\sum_{j=0}^\rho \sum_{k=j}^{\rho}\param^{\rho-k}   \binom{k}{j}\sup_{|\beta|=k-j}\|\partial^\beta\vf\|_{\cC^0}
\sup_{|\gamma|=j}\|\partial^\gamma\psi\|_{\cC^0} \\
&\leq \sum_{j=0}^{\rho}\sum_{l=0}^{\rho-j} \binom{j+l}{j}\param^{\rho-j-l} \sup_{|\beta|=l} \|\partial^\beta\vf\|_{\cC^0}
\sup_{|\gamma|=j}\|\partial^\gamma\psi\|_{\cC^0}\leq \|\vf\|_{\cC^\rho}\|\psi\|_{\cC^\rho}
\end{split}
\]
since $\binom{j+l} j\leq 2^{j+l}\leq \param^{\rho}$. The extension to functions with values in the matrices is trivial since we have chosen a norm in which the matrices form a normed algebra.

To prove the second inequality of the Lemma we proceed again by induction on $\rho$. The case $\rho=1$ is trivial from the definition of the norm. Let us assume that the statement is true for every $k\le \rho$ and show it for $\rho+1$. By the definition of $\|\cdot\|_{\cC^\rho}$,
\begin{equation}
\label{bound | |_rho+1}
\|\varphi\circ \psi\|_{\cC^{\rho+1}}\le \param^{\rho+1}\|\varphi\|_{\cC^0}+ \sup_i\|\partial_{x_i}(\varphi\circ\psi)\|_{\cC^\rho}.
\end{equation}
By hypothesis,\footnote{ Below we use the elementary fact ${{\rho}\choose{q-1}}\le {{\rho+1}\choose{q}}, $  $q\le \rho+1$.} we have
\begin{equation*}
\begin{split}
& \|\partial_{x_i}(\varphi\circ\psi)\|_{\cC^\rho}\le\sup_j \|(\partial_{x_j}\varphi)\circ\psi\|_{\cC^\rho}\|(D\psi)^t\|_{\cC^\rho} \\
&\le \sum_{k=0}^{\rho} {{\rho}\choose{k}} \param^{\rho-k}\|\varphi\|_{\cC^{k+1}}\prod_{i=0}^{k}\|(D\psi)^t\|_{\cC^{\rho-i}}  
\le \sum_{q=1}^{\rho+1} {{\rho+1}\choose{q}} \param^{\rho+1-q}\|\varphi\|_{\cC^{q}}\prod_{j=1}^{q}\|(D\psi)^t\|_{\cC^{\rho+1-j}}. 
\end{split}
\end{equation*}
Finally notice that the term with $q=0$ in the sum above is exactly the first term of the r.h.s. of (\ref{bound | |_rho+1}), which gives the result for $\rho+1$ and proves the induction. 
\end{proof}
\begin{rem}
Notice that, for  $\varphi, \psi \in \mathcal{C}^{\rho}$, the definition of the norm and Lemma \ref{properties of the C norm} imply
\begin{equation}
\label{property of the norm Crho}
\|\varphi \circ \psi\|_{\cC^\rho} \le \|\varphi \|_{\mathcal{C}^{\rho}}\sum_{j=0}^{\rho} {{\rho}\choose{j}} \param^{\rho-j} \|{ (D\psi)^t}\|^{j}_{\mathcal{C}^{\rho -1}}.
\end{equation}
\end{rem}

\section{Foliations: regularity properties}\label{app:fol}
This appendix is devoted to proving Lemma \ref{lem:foliation} and a few other technical Lemmas. In essence we study the behaviour of foliations under iteration. This is very similar to what is done in the construction of the invariant foliations and in the study of their regularity properties, including the regularity of the holonomies. The reason to redo it here without appealing to the literature is that we need these facts in an unconventional form. In particular, we could not find anywhere in the literature the infinitesimal characterisation of the holonomy used here: A characterization hopefully very helpful in the study of discontinuous hyperbolic maps. 

Given such a new twist in the theory, we think it is appropriate to present a more general result: we will control also the regularity of the leaves, and of their tangent spaces, in the unstable direction although this is not needed in the present paper. More precisely we will see that the derivatives of the foliation along the leaves vary in a $\tau_0$-H\"older manner. The optimal $\tau_0$ is well known to depend on a bunching condition \cite{PSW, HW}. We ignore this issue since it largely exceeds our present purposes and to investigate it would entail a lengthier argument. Note that Lemma \ref{lem:foliation} is a special case of Proposition \ref{lem:foliationg} below when choosing $\tau=0$.

Let $\tau\in (0,1)$, given $\vf: M\to\bR$ we define, for some $\delta_\star>0$,
\begin{equation}\label{eq:holder}
\|\vf\|_{\cC^\tau}=\|\vf\|_{\cC^0}+ \sup_{\xi\in M}\sup_{\substack{d(\xi,\xi')\leq \delta_\star\\ \xi\neq \xi'}}\frac{|\vf(\xi)-\vf(\xi')|}{d(\xi,\xi')^\tau}
\end{equation}
where $d(\cdot,\cdot)$ is the Riemannian distance and $\delta_\star\in (0,1)$.  Also, for each $r\in\bR$, $r=q+\tau$, $q\in\bN\cup\{0\}$, we define
\[
\|\vf\|_r=\sum_{k=0}^q\sup_{|\alpha|=k}\|\partial^\alpha \vf\|_{\cC^0}+\sup_{|\alpha|=q}\|\partial^\alpha \vf\|_{\cC^\tau}.
\]
Note that, for $\tau=0$, the above corresponds to \eqref{def of the C norm} with the choice $\param=1$.\footnote{This choice of a different equivalent norm, limited to this appendix, is slightly annoying, but convenient.}

Note that $\|\vf\cdot\phi\|_{\cC^\tau}\leq \|\vf\|_{\cC^\tau}\|\phi\|_{\cC^\tau}$, so $\cC^\tau$ is a Banach algebra. The same holds for matrix valued functions. 

Although the above norms are all equivalent, they depend on $\delta_\star$.   
We will choose $\delta_\star$ in \eqref{eq:deltas}.
Let $T\in\cC^{r}$ and define, for  $\tau \in[0,1)$,
\begin{equation}\label{eq:holder-fol}
\begin{split}
\cW_{\Lb}^{r,\tau}:=\Big\{&W\in\overline{\cF}^r_\cC\;:\;  \sup_{\xi; y\in U^0_s}\|\partial^\alpha_y F_\xi(\cdot,y)\|_{\cC^0(U^0_u,\bR^{d_u})}\leq \Lb^{ (|\alpha|-1)^2},\;  2\leq |\alpha|\leq r;\\ 
&\sup_\xi\sup_{y\in U^0_s}\|\partial^\alpha_y F_\xi(\cdot,y)\|_{\cC^{\tau}(U^0_u,\bR^{d_u})}\leq 2\Lb^{|\alpha|^2},\; |\alpha|\leq r-1;\\
&\sup_\xi\sup_{y\in U^0_s}\|\partial^\alpha_y  H^{F_\xi}(\cdot,y)\|_{\cC^{0}(U^0_u,\bR^{d_s})}\leq \Lb^{(|\alpha|+1)^2},\; |\alpha|\leq r-2;\\
& \sup_\xi\sup_{y\in U^0_s}\|\partial^\alpha_y  H^{F_\xi}(\cdot,y)\|_{\cC^{\tau}(U^0_u,\bR^{d_s})}\leq 2\Lb^{(|\alpha|+2)^2},\; |\alpha|\leq r-3\Big\}.
\end{split}
\end{equation}
Note that, recalling the cone definition \eqref{eq:cones}, the Definition \ref{def:F} and the subsequent definition of $F$, it follows that, for $W\in  \cW_{\Lb}^{r,\tau}$, the corresponding $F$ must satisfy $\|\partial_yF\|\leq 1$ and $\|F(x,y)\|\leq \|x\|+\|y\|$.

\begin{prop} \label{lem:foliationg}
There exists $\tau_0\in (0,1)$,  $\delta_\star >0$, $n_0\in\bN$ and  $\Lb>0$ such that, for all $n\in \bN$, $n\geq n_0$, $\Lb_1\geq \Lb$, $W\in \cW^{r,\tau_0}_{\Lb_1}$, we have $W^n\in  \cW^{r,  \tau_0}_{\Lb_1 /2}$.
\end{prop}
\begin{rem}\label{rem:tauzero} Note that for $\tau=0$ the conditions in $\cW_{L}^{r,\tau}$ reduce to a control on the sup norm of the derivatives $\partial^\alpha_y F(\cdot,y)$ exactly as in the definition of $\cW_{L}^{r}$ in Definition    \ref{def:foliation}. The control stated in Proposition \ref{lem:foliationg} on $\partial^\alpha_y F(\cdot,y)$ is known, as for $\partial^\alpha_y  H^F(\cdot,y)$ we are not aware of this result anywhere in the literature.
\end{rem}
\begin{rem}\label{rem:invariant} Note that for each $W\in \cW_L^{r,\tau}$, $\tau>0$, the foliation $T^nW$ converges to the invariant foliation (since the contraction of the cone fields implies that, for all $x\in M$, $D_{T^{n}x}T^{-n} {\cT}_{T^{n}x} W({T^{n}x})$ converges to the stable distribution $E^s$).\footnote{ Here $\cT_x V$ is the tangent space of the manifold $V$ at the point $x$ and $W(x)$ is the fiber of the foliation passing through $x$. While $E^s(x)$ is the stable subspace in $\cT_x M$.} 
Moreover, if $\bF_n$ describes $T^nW$, then the  $\partial^\alpha_y\bF_n$ are uniformly H\"older.
Accordingly, for each $\tau'<\tau$, by compactness, $T^nW$ has a convergent subsequence, hence it converges to the stable foliation, and all the quantities in the definition of $ \cW_L^{r,\tau}$ converge as well. It follows that the stable foliation have $\cC^r$ leaves with derivatives in $y$ uniformly $\tau'$ H\"older in $x$. Analogously, also $H^F$ and its derivatives converge. This implies that the invariant foliation has a Holonomy uniformly absolutely continuous (see Lemma \ref{lem:holo} and Remark \ref{rem:holo} for the definitions of the Holonomy, its Jacobian $J^F$ and its properties). Similar results hold also in the case $\tau=0$, but the argument is a bit more involved.
\end{rem}
\begin{proof}[{\bf Proof of Proposition \ref{lem:foliationg}}] 
The first step in proving the Proposition is to determine, for each $\xi\in M$ and $n\in\bN$, the functions $F^n_{T^{-n}\xi}$ associated to $W^n$.\footnote{ Since the point $\xi$ in the present argument is fixed once and for all, in the following we will often suppress the subscript $\xi$. We will also suppress the $n$ dependence if no confusion arises.} Note that it suffices to compute the norms in \eqref{eq:holder-fol} in a special neighborhood of $T^{-n}\xi=:\xi'$. 
Indeed, if $\phi_j(\xi')=(x',y')$ and $\widehat U^0_u=\{x\in\bR^{d_u}\;:\; \|x\|\leq \delta_\star\}$ then it suffices to consider the set $(x',y')+\widehat U^0_u\times U^0_s$ since, setting $\zeta_{\xi',x}=\phi_j^{-1}(x' +x, y')$,  a direct computation shows that $F_{\zeta_{\xi',x}}(u,y)=F_{\xi'}(x+u,y)-x$.
Thus
\[
\sup_{\substack{x\in U^0_u\\\|x-\tilde x\|\leq \delta_\star}}\frac{\|F_{\xi'}(x, y)-F_{\xi'}(\tilde x, y)\|}{\|x-\tilde x\|^\tau}=\sup_{x\in U^0_u}\sup_{u\in \widehat U^0_u}\frac{\| F_{\zeta_{\xi',x}}(0, y)-F_{\zeta_{\xi',x}}(u, y)\|}{\|u\|^\tau}.
\]
While, for $|\alpha| >0$,
\[
\begin{split}
&\|\partial^\alpha_y F_{\xi'}(\cdot, y)\|_{\cC^\tau(U^0_u,\bR^{d_u})}=\sup_{x\in U^0_u} \|\partial^\alpha_y F_{\xi'}(x, y)\|+ \sup_{\substack{x\in U^0_u\\\|x-\tilde x\|\leq \delta_\star}}\frac{\|\partial^\alpha_y F_{\xi'}(x, y)-\partial^\alpha_y F_{\xi'}(\tilde x, y)\|}{\|x-\tilde x\|^\tau}\\
&=\sup_{x\in U^0_u}\left\{\|\partial^\alpha_y F_{\zeta_{\xi',x}}(0, y)\|
+ \sup_{u\in \widehat U^0_u}\frac{\|\partial^\alpha_y F_{\zeta_{\xi',x}}(0, y)-\partial^\alpha_y F_{\zeta_{\xi',x}}(u, y)\|}{\|u\|^\tau}\right\}.
\end{split}
\]

 Hence the sup on $y$ and  $\xi'$ can be computed taking the sup of the quantity in the curly bracket (and the same for $H^F$).

 Let $(V_i,\phi_i)$,  $(V_j,\phi_j)$ be the charts associated to $\xi$ and $T^{-n}\xi$ respectively and consider the map $S=\phi_j\circ T^{-n}\circ \phi_i^{-1}$. By a simple translation we can assume, w.l.o.g., that $\phi_i(\xi)=0$ and $\phi_j(T^{-n}\xi)=0$.
From now on we use $(x,y)$ for the coordinate names at $\phi_i(\xi)$ and $(u,s)$ for the coordinate names at $\phi_j(T^{-n}\xi )$.
 By a linear change of coordinates, that leaves $\{y=0\}$ and $\{s=0\}$ fixed, we can have $\partial_yF(0,0)=\partial_{s}F^n(0,0)=0$. Such a change of coordinates may affect the norms yielding some extra (uniformly bounded) constant in the estimates. We will ignore this to simplify notations since its effect is trivial.
Also remember that, by construction, $F(x,0)=x$, $F^n(u,0)=u$.
 
It follows from the usual graph transform (see \cite[Proof of Theorem 6.2.8 (Hadamard-Perron)]{KH}) that 
 \begin{equation}\label{eq:SG}
 S^{-1}(u,0)=(\axp(u), G(\axp(u)),
 \end{equation}
 where $\axp\in\cC^{r+\tau}(\bR^{d_u},\bR^{d_u})$, $\|G\|_{\cC^{r+\tau}(\bR^{d_u},\bR^{d_s})}\leq c_1$, for some $c_1>0$ depending only on $T$. Moreover, $\|DG\|\leq \srk< 1$ by the invariance of the cone field and $\axp(0)=0, G(0)=0$ by construction. Moreover, by \eqref{eq:anosov},  $\|(D\axp)^{-1}\|_{\cC^{0}}\leq \czero^{-1}\lambda^{-n}$ while, setting $\lambda_+=\max\{\|DT\|, \|DT^{-1}\|\}$, we have, for some constant $C_1>0$,
\begin{equation}\label{eq:alphasup}
\|\axp(u)\|\leq \|(\axp(u),G\circ \axp(u))\|\leq C_1\lambda_+^n\|u\|.
\end{equation}

In addition, $\{(z,D_\xi G(z))\}_{z\in\bR^{d_u}}$ is uniformly traversal to $\{(D_\xi F(\zeta),\zeta)\}_{\zeta\in\bR^{d_s}}$. 

Hence, setting
\begin{equation}\label{eq:der-mat}
D_{(x,y)}S=\begin{pmatrix} A& B\\C&E\end{pmatrix}
\end{equation}
we have
\begin{equation}\label{eq:derS}
\begin{split}
&A(x,G(x))=(D_x\axp)^{-1}- B(x,G(x))D_xG\\
&C(x,G(x))=-E(x,G(x))D_xG\;\;;\quad B(0,0)=0,
\end{split}
\end{equation}
where the last equality follows by the choice of the coordinates.

For each $x$, the manifold $\{(F(x,y),y)\}_{y\in\bR^{d_s}}$ intersects the manifold $\{(z, G(z))\}_{\bR^{d_u}}$ in a unique point determined by the equation
\[
(F(x,y),y)=(z, G(z))
\]
which is equivalent to $L(z,x):=z-F(x,G(z))=0$. Since $L(0,0)=0$ we apply the implicit function theorem and obtain a function $\Gamma: U^0_u\subset \bR^{d_u}\to \bR^{d_u}$ such that~\footnote{\label{foo:delta0} Since the implicit function theorem yields a uniform domain $D(\Gamma)$, of $\Gamma$, we can take $\delta_0$ small enough so that $D(\Gamma)\supset  U^0_u$.}
\begin{equation}\label{eq:gamma}
\Gamma(x)=F(x,G\circ\Gamma(x)).
\end{equation}
Since $\|\partial_yF\|\leq 1$ and recalling  $\|DG\|\leq \srk<1$,  $\Id-\partial_yF DG$ is invertible. Hence,
\[
D\Gamma=(\Id-\partial_yF DG)^{-1}\partial_xF.
\]
Note that, remembering \eqref{eq:F-def}, $D_0\Gamma=\Id$. Moreover, $\Gamma$ is invertible since, recalling \eqref{eq:gamma}, $\Gamma(x) =\Gamma(x')$ implies 
\[
(F(x,G\circ\Gamma(x)), G\circ\Gamma(x))=(F(x',G\circ\Gamma(x')), G\circ\Gamma(x))=(F(x',G\circ\Gamma(x)), G\circ\Gamma(x))
\]
which forces $x=x'$ since the leaves of the foliation are disjoint by hypothesis.

By definition, for each $u\in \bR^{d_u}$ small enough, $\{(F^n(u,s),s)\}_{s\in \bR^{d_s}}$ is the graph of the leaf of $W^n$ passing through $(u,0)$, hence of the image of the leaf of $W$ passing through $(\Gamma^{-1}\circ \axp(u),0)$. In other word  $\{(F^n(u,s),s)\}_{s\in \bR^{d_s}}$ coincides with the leaf $\{S(F(\Gamma^{-1}\circ\axp(u), y),y)\}_{y\in \bR^{d_s}}$.

To continue we need some estimates on $DS$. But, before that, it is convenient to make some choices and definitions whose meaning will become clear later in the proof. Let  $ \tau_0\in(0, 1)$ be such that
\begin{equation}\label{eq:lambdas}
\sigma_1:=\max\{\nu, \lambda^{-1}\}\cdot\lambda_+^{ 8\tau_0}< 1.
\end{equation}
Next, let $\sigma_1<\sigma<1$, fix $C_\star>0$ to be chosen later (see equations \eqref{eq:Phis},  \eqref{eq:L1-chose0} and \eqref{eq:estimate1C_star}) and let $n_\star$ be the smallest integer such that
\begin{equation}\label{eq:cstar}
C_\star\sigma^{n_\star}_{1}=\sigma^{n_\star}< (1/8)^{r}\;,\quad\nu^{n_\star}\leq 2^{-2r}.
\end{equation}
\begin{rem}\label{rem:delta0} Up to now $\delta_0$ was arbitrary provided we chose it small enough: the requirements are in Section \ref{sec:foliation} where we fix the charts and in footnote \ref{foo:delta0}. In the following we will have also a condition in equation \eqref{eq:FPhi} to apply the implicit function theorem, and we will use $\delta_0<1/8$ in equation \eqref{eq:Fn}. All such choices can be sumarized by the condition $\delta_0\leq \delta_1$  for some $\delta_1\in (0,1/8)$ depending only on $T$. However in the next Lemma we will have a requirement depending on $n$.
\end{rem}

\begin{lem}\label{eq:S-est}
There exists $\srk, \sigma_0\in (0,1)$ and $\Czero\geq\max\{2, 6\czero^{-1}\}$ such that, for each $n\in \bN$, $d_u\times d_s$ matrix $U$, $\|U\|\leq 1$, we have
\[
\begin{split}
&\|CU+E\|\geq \Czero^{-1}\nu^{-n}\;;\quad  \|E^{-1}C\|\leq \srk \;;\quad  \|(AU+B)(CU+E)^{-1}\|\leq \srk.
\end{split}
\]
In particular, $\|E^{-1}\|\leq \Czero\nu^{n}$. 
Moreover, there exists a constant $C_\flat>0$ such that, if we choose $\delta_0=\min\{\delta_1, \Czero C_\flat^{-1}\lambda_+^{-4n_\star}/3\}$, then,  for all $n\in\{n_\star, \dots,2n_\star\}$,  $\|A\|+ \|B\|\leq \Czero\lambda^{-n}$.
\end{lem}
\begin{proof}
By the strict invariance of the cone field (see section \ref{sec:map}) and the Anosov property \eqref{eq:anosov} it follows that, for each $U$ there exists matrices $U_1, H$, with $\|U_1\|\leq \srk$ and $\|H\| \geq \frac{\czero\sqrt 2}{\sqrt{1+\srk^2}}\nu^{-n}$, such that
\[
(U_1Hv, Hv)=DS\begin{pmatrix} Uv\\ v\end{pmatrix}=([AU+B]v,[CU+E]v).
\]
Thus, $H=CU+E$ and $U_1=(AU+B)(CU+E)^{-1}$ from which the first and third inequality readily follow. Analogously, for each $V$ there exists $\tilde V, \tilde H$, $\|\tilde V\|\leq \srk$, $\|\tilde H\|\leq \czero^{-1}\sqrt 2\lambda^{-n}$, such that
\[
DS\begin{pmatrix} v\\ \tilde Vv\end{pmatrix}=(\tilde H v,V\tilde H v),
\]
which implies the second inequality and, for $V=0$, yields $\|E^{-1}C\|=\|\tilde V\|\leq \srk$. 
Note that
\[
\begin{split}
&\partial_{x_p}DT^{-n}=D\left(D_{T^{-n+1}x}T^{-1}\cdots D_xT^{-1}\right)\\
&=\sum_{k=0}^{n-1}\sum_{j=1}^d D_{T^{-n+1}x}T^{-1}\cdots D_{T^{-k-1}x}T^{-1} \partial_{x_j}(D_{T^{-k}x}T^{-1})D_{T^{-k+1}x}T^{-1}\cdots D_xT^{-1}\\
&\phantom{=} \times (D_xT^{-k})_{jp}
\end{split}
\]
Recalling that $\lambda_+=\max\{\|DT\|, \|DT^{-1}\|\}$, it follows that there exist a constant $C_\flat>0$, depending only on $T$ and on the coordinate changes $\{\phi_i\}$, such that $\|DB\|\leq C_\flat\lambda_+^{2n_\star}$. 
Then, since $B(0,0)=0$, see \eqref{eq:derS}, it follows
 \[
 \|B\|\leq \delta_0\|DB\|\leq \Czero \lambda_+^{-2n_\star}/3\leq \Czero \lambda^{-n}/3,
 \]
 and
\[
\|A\|\leq \|\tilde H\|+\|B\tilde V\|\leq  \|\tilde H\|+\srk \Czero \lambda^{-n}/3\leq \left(\czero^{-1}\sqrt 2 \lambda^{-n}+\Czero \lambda^{-n}/3\right)\leq \frac{2\Czero}{3}\lambda^{-n}
\]
from which the last assertion of the Lemma readily follows.
\end{proof}

We are now ready to study $F^n$. Let us consider the function
\[
\Xi(v,s,u,y)=(v,s)-S(F(\Gamma^{-1}\circ\axp(u), y),y).
\]
It is convenient to set $\Upsilon(u)=\Gamma^{-1}\circ \axp(u)$. Note that \eqref{eq:gamma} implies
\[
\|x\|=\|F(\Gamma^{-1}(x),G(x))\|\geq \|F(\Gamma^{-1}(x),0)\|-\srk\|x\|=\|\Gamma^{-1}(x)\|-\srk\|x\|,
\]
That is
\begin{equation}\label{eq:gamma-u}
\|\Gamma^{-1}(x)\|\leq (1+\srk)\|x\|.
\end{equation} 
We want to insure that $\|\Upsilon(u)\|\leq \delta_0$. Recalling \eqref{eq:alphasup} this is implied by
\begin{equation}\label{eq:deltas}
C_1(1+\srk)\delta_\star\lambda_+^{2n_\star} \leq\delta_0.
\end{equation}
Note that 
\[
\begin{split}
\Xi(u,0,u,G(\axp(u)))&=(u,0)-S(F(\Upsilon(u),G(\axp(u)) ),G(\axp(u)))\\
&=(u,0) -S(\axp(u),G(\axp(u)))=0.
\end{split}
\]
To study the zeroes of $\Xi$ we apply the implicit function theorem. Since
\[
\det\begin{pmatrix}\partial_v \Xi&\partial_y\Xi\end{pmatrix}
=\det\begin{pmatrix}
\Id& A\partial_y  F+B\\
0&C\partial_yF +E
\end{pmatrix}
\]
we can compute
\[
\begin{split}
&\det \begin{pmatrix}
\partial_v \Xi (u,0,u,G(\axp(u)))&\partial_y\Xi(u,0,u,G(\axp(u)))
\end{pmatrix}\\
&=\det\begin{pmatrix} E(\axp(u),G(\axp(u)))\end{pmatrix}\det\begin{pmatrix} \Id-DG(\axp(u))\partial_y F(\Upsilon(u)) \end{pmatrix}\neq 0
\end{split}
\]
where we have used \eqref{eq:derS} and $\|DG\partial_y F\| \leq \srk$. Thus there exists a uniform (in $n$) neighborhood of $(u, G(\axp(u))$ where the implicit function theorem can be applied. Thus, we can choose $\delta_0$ small enough so that, for each $n\in\bN$, there exist $F^n,\Phi\in \cC^r$:
\begin{equation}\label{eq:FPhi}
\begin{split}
&v=F^n(u,s)\\
&y=\Phi(u,s).
\end{split}
\end{equation}
Moreover, defining the change of coordinates $\Omega(u,s)=(\Gamma^{-1}\circ \axp(u),\Phi(u,s))=(x,y)$,
\begin{equation}\label{eq:f-transf}
\bF^n(u,s)=S\circ \bF\circ\Omega(u,s)\,.
\end{equation}
Note that
\begin{equation}\label{eq:Fn}
\|\bF^n(u,s)\|\leq \|\bF^n(u,s)-\bF^n(u,0)\|+\|\bF^n(u,0)\|\leq \|s\|+\|u\|\leq 2\delta_0\leq \frac{1}{4}.
\end{equation}
Differentiating \eqref{eq:f-transf} with respect to $s$ we obtain
\[
\begin{pmatrix}
\partial_uF^n& \partial_s F^n\\  0&\Id
\end{pmatrix}
=\begin{pmatrix} A&B\\ C&E \end{pmatrix}\begin{pmatrix} \partial_x F&\partial_yF\\ 0&\Id \end{pmatrix} \begin{pmatrix}  \partial_u\Upsilon&0\\\partial_u\Phi&\partial_s\Phi\end{pmatrix}
\]
which yields, 
\begin{equation}\label{eq:s-der}
\begin{split}
&\partial_s\Phi=(E\circ \bF\circ \Omega+C\circ \bF\circ \Omega\cdot \partial_yF\circ \Omega)^{-1} \\
&\partial_sF^n=\left(A\circ \bF\circ \Omega\cdot \partial_yF\circ \Omega+B\circ \bF\circ \Omega\right) \partial_s\Phi.
\end{split}
\end{equation}
Then Lemma \ref{eq:S-est} yields, provided $C_\star\geq \frac{ \Czero^2}{1-\srk}$,
\begin{equation}\label{eq:Phis}
\begin{split}
&\|\partial_s\Phi\|\leq \|E^{-1}\|\|(\Id+[E^{-1} C]\circ \bF\cdot\partial F)^{-1}\|\leq \frac{ \Czero}{1-\srk}\nu^n\leq \frac 1{8^r}\\
& \|\partial_sF^n\|\leq \frac{ \Czero^2}{1-\srk}\nu^n\lambda^{-n}\leq \frac 1{8^r}.
\end{split}
\end{equation}
We now study $\partial_s^{\alpha}F^n$ when $|\alpha|\ge 2$. Differentiating \eqref{eq:s-der} and setting $\Delta_-=(\Id+(E^{-1}C)\circ \bF\cdot \partial_yF)^{-1}$ and $\Delta_+=(A\circ \bF\cdot \partial_yF+B\circ \bF)$ we obtain, 
\begin{equation}\label{eq:F-der2}
\begin{split}
\partial_{s_2}\partial_{s_1} F^n=&\sum_{i_1,i_2=1}^{d_s} \cH(DS\circ F,\partial_yF,  \partial_{y_{i_2}}\partial_{y_{i_1}}F)\circ \Omega\cdot  \partial_{s_1}\Phi_{i_1} \partial_{s_2}\Phi_{i_2}\\
&+\Theta_{1}(DS\circ F, D^2S\circ F,  \partial_y F)\circ \Omega
\end{split}
\end{equation}
where $\Theta_{1}$ is a  rational function of its arguments and, for an arbitrary matrix $\cR$,
\[
\begin{split}
\cH(DS\circ F,\partial_yF, \cR):=&\left\{A\circ \bF\cdot \cR\cdot \Delta_- \cdot E^{-1}\circ \bF\right\} \\
&-\left\{\Delta_+\Delta_-\cdot ( E^{-1}C)\circ \bF\cdot \cR\cdot\Delta_-\cdot E^{-1}\circ \bF\right\}.
\end{split}
\]
Note that Lemma \ref{eq:S-est} implies
\begin{equation}\label{eq:cH}
\|\cH(DS \circ F,\partial_y F,\cR)\|\leq \frac{\Czero^2}{(1-\srk)^2}\lambda^{-n}\nu^n\|\cR\|.
\end{equation}

Differentiating further \eqref{eq:F-der2} we can prove, by induction,  that, for all $l\leq r$, 
\begin{equation}\label{eq:Phis-der}
\begin{split}
\partial_{s_{j_l}}\dots\partial_{s_{j_1}} F^n=&\!\!\!\!\!\sum_{i_1,\dots, i_l=1}^{d_s}\!\!\!\!\! \cH(DS\circ F,\partial_yF, \partial_{y_{i_l}}\dots \partial_{y_{i_1}}F)\circ \Omega\cdot  \partial_{s_{j_1}}\Phi_{i_1}\cdots  \partial_{s_{j_l}}\Phi_{i_l}\\
&+\Theta_{l-1}(DS\circ F, \dots, D^lS\circ  F, \partial_y F, \dots, \partial_y^{l-1}F)\circ \Omega
\end{split}
\end{equation}
where the $\Theta_{l}$ are sums of functions $k_j$-multilinear in $\partial_y^jF$, for $j\in\{2,\dots, l\}$,  such that $\sum_{j=2}^{l} k_j(j-1)\leq l$.
Indeed, we have seen that this is true for $l=2$. On the other hand, if it is true for $l-1$, then differentiating \eqref{eq:Phis-der} we produce several terms. Let us analyse them one by one.  The term proportional to $\cH$, when differentiated with respect to $\partial_y^l F$,  yields the correct term proportional of $\cH$. When differentiated with respect to $D^kS\circ F$ yields a function of $D^{k+1}S\circ F$ multiplied by  $\partial_yF\cdot \partial_s\Phi$ so the multilinearity with respect to $\partial_y^jF$, for $j\in\{2,\dots, l\}$, is unchanged. When differentiating with respect to $\partial_yF$, the term gets multiplied by $\partial_y^2F$.\footnote{ Recall that a dependence from $\partial_yF$ is contained in $\partial_s\Phi$, see \eqref{eq:s-der}.}
Thus, calling $k_j'$ the multilinearities of the term obtained we have $k_l'=1$, $k_2'=1$ and all the other $k_j'$ are zero. That is 
$\sum_{j=2}^l k_j'(j-1)=l-1+1=l$.

 Next we must differentiate $\Theta_{l-1}$. Again the only change in the multilinearity occurs when differentiating with respect to a $\partial_y^m F$, $m\in\{1,\dots, l-1\}$. If $j=2$ then we have (calling again $k'$ the new multilinearities) $k_2'=k_2+1$ and $k_j'=k_j$ for $j>2$, that is $\sum_{=2}^{l-1}k_j'(j-1)=l-1+1=l$. If $m>1$, hence $k_m'=k_m-1$,  $k_{m+1}'=k_{m+1}+1$ and $k_j'=k_j$ for $j\not\in\{m, m+1\}$, that is  $\sum_{ j=2}^lk_j'(j-1)=l-1-(m-1)+m=l$. Which proves our claim.
 
Using equations \eqref{eq:Phis} and \eqref{eq:cH} to estimate \eqref{eq:Phis-der} yields, for all $l\in\{2,\dots, r\}$,
\[
\begin{split}
\|\partial^l_s F^n\|&\leq d_s^l \frac{\Czero^{2+l}}{(1-\srk)^{2+l}}\lambda^{-n}\nu^{n(l+1)}\|\partial^l_s F\|+\|\Theta_{l-1}\|_\infty\\
&\leq d_s^l \frac{\Czero^{2+l}}{(1-\srk)^{2+l}}\lambda^{-n}\nu^{n(l+1)}\Lb_1^{(l-1)^2}+C_n \Lb_1^{\sum_{j=2}^{l-1} k_j(j-1)^2}\\
&\leq \left[d_s^l \frac{\Czero^{2+l}}{(1-\srk)^{2+l}}\lambda^{-n}\nu^{n(l+1)}+C_n \Lb_1^{-(l-1)}\right]\Lb_1^{(l-1)^2}.
\end{split}
\]
Choosing
\begin{equation}\label{eq:L1-chose0}
\begin{split}
&C_\star\geq d_s^{r} \frac{\Czero^{2+ r}}{(1-\srk)^{2+r}}\\
&\Lb_1>\max\{1, 2^{r^2} \max\{C_{n_\star},\dots, C_{2n_\star}\}\}.
\end{split}
\end{equation}
equation \eqref{eq:cstar} implies, as announced,
\begin{equation}\label{eq:Fnbound}
\|\partial^l_s F^n\|\leq [\Lb_1/2]^{(l-1)^2}.
\end{equation}
Next we estimate the H\"older norms of $\partial^\alpha_s\bF^n$ for $|\alpha|\leq r-1$. We first treat the case $|\alpha|=0$.
By strict cone field invariance and the continuity of the cone field it follows that, for all $s$, 
\begin{equation}\label{eq:Sphi}
S^{-1}\bF^n(u,s)=(F\circ \Omega(u,s), G_s(F\circ \Omega(u,s)))
\end{equation}
with $\|DG_s\|\leq \srk<1$. Remark that \eqref{eq:f-transf}, \eqref{eq:Sphi} and \eqref{eq:SG} imply 
\[
\begin{split}
(F\circ \Omega(u,0), \Phi(u,0))&=(F\circ \Omega(u,0), G_0(F\circ \Omega(u,0)))\\
&=S^{-1}\bF^n(u,0)=S^{-1}(u,0)=(\axp(u), G(\axp(u))),
\end{split}
\]
that is $G_0=G$ and 
\begin{equation}\label{eq:PG}
\Phi(u,0)=G(\axp(u)).
\end{equation}
Analogously, by \eqref{eq:f-transf} $S^{-1}\bF^n(u,s)=(F\circ \Omega(u,s),\Phi(u,s))$. Hence
\[
\begin{split}
\|\Phi(u,s)-&\Phi(u',s)\|=\|G_s(F\circ \Omega(u,s))-G_s(F\circ \Omega(u',s))\|\\
&\leq \srk \|F\circ \Omega(u,s)-F\circ \Omega(u',s)\|\\
&\leq \srk \|F(\Upsilon(u), \Phi(u,s))-F(\Upsilon(u'), \Phi(u,s))\|+\srk \|\Phi(u,s)-\Phi(u',s)\|
\end{split}
\]
where we have used $\|\partial_yF\|\leq 1$. Accordingly
\begin{equation}\label{eq:phi-u}
\|\Phi(u,s)-\Phi(u',s)\|\leq \frac{\srk}{1-\srk}\|F(\Upsilon(u), \Phi(u,s))-F(\Upsilon(u'), \Phi(u,s))\|.
\end{equation}

Next we prove an auxiliary lemma, which will be used repeatedly in the following.  
\begin{lem}\label{lem:aux}
Let $\cG:U^0\to\bR^{d'}$, $d'\in\bN$. Assume $\sup_{y\in U^0_s}\|\cG(\cdot,y)\|_{\cC^0(U^0_u,\bR^{d'})}\le D_0$,\newline $\sup_{y\in U^0_s}\|\cG(\cdot,y)\|_{\cC^\tau(U^0_u,\bR^{d'})}\le D'$ with $\tau\in[0,\tau_0]$ and $\|\partial_y \cG\|_{\cC^0}\le \tilde D$.  Then, for all $n\in\{n_\star,\dots, 2n_\star\}$, we have
\[
 \czero^{-1}\max\{\nu^n,\lambda^{-n}\}\| \cG\circ\Omega(u,s)- \cG\circ\Omega(0,s)\|\leq \frac{1}{8} \left[ (1-\srk)\max\{D_0,D'\} +2\srk \tilde D\right]\|u\|^{\tau}.
\]
\end{lem}
\begin{proof}
Let $\theta=\max\{\nu,\lambda^{-1}\}$. We start analyzing $\|\Upsilon(u)\|\leq \delta_\star$.   By \eqref{eq:phi-u}, we get
  \[
 \|\Phi(u,s)-\Phi(0,s)\|\leq \frac{2\srk }{(1-\srk)}\|\Upsilon(u)\|^{\tau_0} .
\]
Hence, by \eqref{eq:gamma-u} and \eqref{eq:alphasup} we have,
\[
\begin{split}
&  \czero^{-1}\theta^{n}\| \cG\circ\Omega(u,s)- \cG\circ\Omega(0,s)\| 
\le  \czero^{-1}\theta^{n}\left[D'\|\Upsilon(u)\|^\tau+\tilde D\|\Phi(u,s)-\Phi(0,s)\|\right]\\
&\le  \frac{\theta^{n}(1+\srk)^{\tau}}{\czero }\left[ D'+\frac{2\tilde D \srk}{1-\srk}\right]\|\axp(u)\|^\tau
\leq \frac{\theta^{n}(1+\srk)^{\tau}}{\czero }\left[ D'+\frac{2\tilde D \srk}{1-\srk}\right]C_1^\tau \lambda_+^{\tau n}\|u\|^\tau.
\end{split}
\]
Consequently, provided that $C_\star$ in \eqref{eq:cstar} satisfies\footnote{ See remark \ref{rem:delta0} for the definition of $\delta_1$. Also we consider a $C_\star$ larger than what is needed at this stage for later purposes in this proof. }
\begin{equation}\label{eq:estimate1C_star}
C_\star \geq \frac{2(1+\srk)^{\tau} C_1^\tau (\delta_1^{-\tau_0}+C_\flat^\tau)}{\czero{(1-\srk)}},
\end{equation}
we have
\begin{equation}
\label{eq:DG}
\begin{split}
 \czero^{-1}\theta^{n}&\| \cG\circ\Omega(u,s)- \cG\circ\Omega(0,s)\|\leq \left[C_\star(1-\srk) D'+2\srk C_\star\tilde D \right]\sigma_1^n\|u\|^{\tau}\\
&\le \frac{1}{8} \left[ (1-\srk)D' +2\srk \tilde D\right]\|u\|^{\tau},
\end{split}
\end{equation}
where, in the last line, we have used \eqref{eq:cstar}. We are left with the analysis of the case $\|\Upsilon(u)\|\geq \delta_\star$.
By  \eqref{eq:gamma-u},  \eqref{eq:deltas}, \eqref{eq:alphasup}, \eqref{eq:lambdas} we have 
\begin{equation*}
\begin{split}
 \czero^{-1}\theta^{n}\| \cG\circ\Omega(u,s)- \cG\circ\Omega(0,s)\|&\leq \czero^{-1}\theta^{n} 2D_0
\leq \czero^{-1}\theta^{n}(1+\srk)^\tau 2D_0\delta_\star^{-\tau}\|\axp(u)\|^{\tau} \\
&\leq  \czero^{-1}\theta^{n}(1+\srk)^\tau 2D_0\delta_0^{-\tau}\lambda_+^{2\tau n_\star}
C_1^\tau \lambda_+^{\tau n}\|u\|^{\tau}.
\end{split}
\end{equation*}
Thus, if $\delta_0=\delta_1$, then
\[ \czero^{-1}\theta^{n}\| \cG\circ\Omega(u,s)- \cG\circ\Omega(0,s)\|\leq  \frac{(1-\srk)D_0}{8} \|u\|^{\tau},
\]
otherwise, recalling Lemma \ref{eq:S-est} and equation \eqref{eq:cstar},
\begin{equation}\label{eq:DG2}
\begin{split}
 \czero^{-1}\theta^{n}\| \cG\circ\Omega(u,s)- &\cG\circ\Omega(0,s)\|\leq \frac{(1-\srk)3^\tau C_\flat^\tau C_\star}{\Czero^\tau} \delta_0^{-\tau}\theta^{n_\star} \lambda_+^{4n_\star\tau}D_0\|u\|^{\tau}\\
&\leq   (1-\srk)C_\star \theta^{n_\star} \lambda_+^{8n_\star\tau}D_0\|u\|^{\tau}\leq  \frac{ (1-\srk)D_0}{8} \|u\|^{\tau}
\end{split}
\end{equation}
from which the Lemma follows.

Note, for further use, that the above computation imply, as well,
\begin{equation}\label{eq:phi-est}
\czero^{-1}\theta^n\|\Phi(u,s)-\Phi(0,s)\|\leq \frac{1}{2} \|u\|^{\tau_0}.
\end{equation}

\end{proof}
We now estimate $\|\bF^n(u,s)-\bF^n(0,s)\|$.
By \eqref{eq:anosov} and \eqref{eq:f-transf} we have, 
\[
\|\bF^n(u,s)-\bF^n(0,s)\|\leq \czero^{-1}\lambda^{-n} \| (F\circ\Omega(u,s), \Phi(u,s))- (F\circ\Omega(0,s), \Phi(0,s))\|.
\]
Recalling \eqref{eq:phi-est}, we can apply Lemma \ref{lem:aux} with $\cG=F$, $\tau=\tau_0$, $D_0=1$, $D'=2$ and $\tilde D=1$ to obtain, for all $n\in[n_\star, 2n_\star]$,
\begin{equation}\label{eq:Ftau}
\|\bF^n(u,s)-\bF^n(0,s)\|\leq \frac{1}{4}\|u\|^{\tau_0}.
\end{equation}
From which, recalling \eqref{eq:Fn}, it follows the wanted estimate
\begin{equation}\label{eq:h-1}
\sup_s\|\bF^n(\cdot, s)\|_{\cC^{\tau_0}}\leq \frac 12.
\end{equation}
Next, we discuss the case $|\alpha|>0$. 
We start by estimating $\|\partial_{s}\Phi\|_{\cC^\tau_0}$. To simplify the notation in the expression below, let $a=(u,s)$ and $b=(0,s)$.
Using \eqref{eq:s-der}, \eqref{eq:Phis} and Lemma \ref{eq:S-est} notice that
\[
\begin{split}
&\|\partial_s\Phi(a)-\partial_s\Phi(b)\|\\
&\le
 \|E^{-1}\circ \bF\circ \Omega(b)-E^{-1}\circ \bF\circ \Omega(a)\|\cdot\|\big(\Id+E\circ \bF\circ \Omega(b)^{-1}C\circ \bF\circ \Omega(b)\cdot \partial_yF\circ \Omega(b)\big)^{-1}\|\\
 &+\|(E^{-1}\circ \bF\circ \Omega(a))\|\cdot\|\big(\Id+E\circ \bF\circ \Omega(a)^{-1}C\circ \bF\circ \Omega(a)\cdot \partial_yF\circ \Omega(a)\big)^{-1}\|\\
& \times\|E\circ \bF\circ \Omega(a)^{-1}C\circ \bF\circ \Omega(a)\cdot \partial_yF\circ \Omega(a)-E\circ \bF\circ \Omega(b)^{-1}C\circ \bF\circ \Omega(b)\cdot \partial_yF\circ \Omega(b)\|\\
&\times\|\big(\Id+E\circ \bF\circ \Omega(b)^{-1}C\circ \bF\circ \Omega(b)\cdot \partial_yF\circ \Omega(b)\big)^{-1}\|\\
&\le
 \|E^{-1}\circ \bF\circ \Omega(b)-E^{-1}\circ \bF\circ \Omega(a)\|\times\|\big(\Id+E\circ \bF\circ \Omega(b)^{-1}C\circ \bF\circ \Omega(b)\cdot \partial_yF\circ \Omega(b)\big)^{-1}\|\\
 &+ \frac{\Czero\nu^n}{1-\srk}\|E\circ \bF\circ \Omega(a)^{-1}C\circ \bF\circ \Omega(a)\cdot \partial_yF\circ \Omega(a)-E\circ \bF\circ \Omega(b)^{-1}C\circ \bF\circ \Omega(b)\cdot \partial_yF\circ \Omega(b)\|\\
 &\le \frac{\Czero\nu^n}{1-\srk}\|\partial_yF\circ \Omega(a)- \partial_yF\circ \Omega(b)\|+C_n\nu^{n}\|\bF\circ \Omega(b)-\bF\circ \Omega(a)\|.
\end{split}
\]
We apply Lemma \ref{lem:aux} with $\cG=\partial_yF$, $\tilde D=\Lb_1$, $D_0=1$, $D'=2\Lb_1$ and with $\cG=F$ with $\tilde D=1$, $D_0=1$, $D'=2$. Recalling \eqref{eq:phi-est}, we obtain
\begin{equation}\label{eq:phi-hold}
 \|\partial_s\Phi(a)-\partial_s\Phi(b)\|\le \left[\frac{\Czero\czero}{(1-\srk)4} \Lb_1+\frac{C_n\czero}{4}\right] \|u\|^{ \tau}.
\end{equation}

Next, by \eqref{eq:s-der},
\[
\begin{split}
\|\partial_sF^n(a)-\partial_sF^n(b)\|\le &\|\left[A\circ \bF\cdot \partial_yF+B\circ \bF\right]\circ \Omega(b)-\left[A\circ \bF\cdot \partial_yF+B\circ \bF\right]\circ \Omega(a)\|\cdot\|\partial_s\Phi\|_{\cC^0}\\
&+\|\left[A\circ \bF\cdot \partial_yF+B\circ \bF\right]\circ \Omega\|_{\cC^0}\cdot\| \partial_s\Phi(b)- \partial_s\Phi(a)\|.
\end{split}
\]
Using again Lemma \ref{eq:S-est}, 
\[
\begin{split}
& \|A\circ \bF\circ \Omega(a)\cdot \partial_yF\circ \Omega(a)+B\circ \bF\circ \Omega(a)-A\circ \bF\circ \Omega(b)\cdot \partial_yF\circ \Omega(b)-B\circ \bF\circ \Omega(b)\|\\
&\leq \|A\circ \bF\circ \Omega(a)\|\| \partial_yF\circ \Omega(a)-\partial_yF\circ \Omega(b)\| + C_n\|\bF\circ \Omega(b)-\bF\circ \Omega(a)\|\\
& \le \Czero\lambda^{-n}\| \partial_yF\circ \Omega(a)-\partial_yF\circ \Omega(b)\| + C_n\nu^n\|\bF\circ \Omega(b)-\bF\circ \Omega(a)\|.
\end{split}
\]
Arguing as above and remembering \eqref{eq:Phis}, \eqref{eq:lambdas} and \eqref{eq:cstar} we obtain, for all $n\in\{n_\star, \dots, 2n_\star\}$,
\begin{equation}\label{eq:almost-last}
\|\partial_sF^n\|_{\cC^\tau}\le 1+\sigma_1^{n}\Czero^2\czero \Lb_1+\czero C_n<2\Lb_1,
\end{equation}
provided 
\begin{equation}\label{eq:L1-chose}
\Lb_1>\max\{1, 2\czero \max\{C_{n_\star},\dots, C_{2n_\star}\}\}.
\end{equation}
For estimating the H\"older constant of $\partial^\alpha_sF^n$, $|\alpha|\in\{ 2,\dots, r-1\}$, we can use \eqref{eq:Phis-der}.\footnote{Since the argument uses a bound on $\partial_s^\beta F^n$ for $|\beta|=|\alpha|+1$, we stop at $|\alpha|\leq r-1$.} Indeed, recalling \eqref{eq:Phis}, \eqref{eq:cH} and arguing similarly to before yields
\begin{equation}\label{eq:Fholder}
\begin{split}
\|\partial_s^\alpha F^n(a)-\partial_s^\alpha F^n(b)\|&\le \left\{\left[\frac{\Czero}{1-\srk}\theta^n\right]^{|\alpha|+2}\!\!
\! \Lb_1^{|\alpha|^2}+\sum_{j=0}^{|\alpha|-1} \Const\Lb_1^{(|\alpha|-1)^2-(j-1)^2}\Lb_1^{j^2}\right\}\|u\|^\tau\\
&\leq \left\{ \left[\frac{\Czero}{1-\srk}\theta^n\right]^{|\alpha|+2} + \Const\Lb^{-2}\right\} \Lb_1^{|\alpha|^2}\|u\|^\tau\leq (\Lb_1/2)^{|\alpha|^2}\|u\|^\tau
\end{split}
\end{equation}
provided $\Lb_1$ has been chosen large enough. To obtain the estimate for all $n\in\bN$ it suffices to write $k=kn_\star+m$, with $m\in\{n_\star, \dots,2n_\star\}$ and then iterating the inequalities.

We are left with the study of $H^F$.
Recalling \eqref{eq:der-mat}, \eqref{eq:derS}, \eqref{eq:s-der} and differentiating \eqref{eq:f-transf} with respect to $u$,
yields
\begin{equation}\label{eq:u-der}
\begin{split}
&\partial_u\Phi=-(E\circ \bF\circ \Omega+C\circ \bF\circ \Omega\cdot \partial_yF\circ \Omega)^{-1}C\circ \bF\circ \Omega\cdot \partial_xF\circ \Omega\cdot D\Upsilon\\
&\partial_uF^n= \left[A\circ \bF-(A\circ \bF\cdot \partial_yF+B\circ \bF)(E\circ \bF+C\circ \bF\cdot\partial_yF)^{-1}C\circ \bF\right]\circ\Omega\cdot \partial_x F\circ \Omega\cdot D\Upsilon\\
&\quad\quad\;\,=\left[A\circ \bF\circ\Omega-\partial_sF^n \cdot C\circ \bF\circ\Omega\right]\cdot \partial_x F\circ\Omega \cdot D\Upsilon=:\left[\Lambda\cdot \partial_x F\circ\Omega\right]\cdot D\Upsilon.
\end{split}
\end{equation}
We can now compute,
\[
\begin{split}
H^{F^n}_l\circ (\bF^n)^{-1}=&\sum_i\partial_{u_i}\left[(\partial_{s_l} F^n_i)\circ (\bF^n)^{-1}\right]=\sum_{i,k}\left[\partial_{s_l}\partial_{u_k} F^n_i\cdot (\partial_u F^n)^{-1}_{k,i}\right]\circ (\bF^n)^{-1}\\
=&\tr\left([\partial_{s_l}\partial_{u} F^n] (\partial_u F^n)^{-1}\right)\circ (\bF^n)^{-1}.
\end{split}
\]
Thus, using \eqref{eq:u-der},
\begin{equation}\label{eq:hol-h}
\begin{split}
H^{F^n}_l&=\tr\big[(\partial_{s_l}\Lambda)\cdot \partial_x F\circ \Omega\cdot D\Upsilon (\partial_u F^n)^{-1}\\
&\phantom{=}
+\Lambda\cdot (\partial_{s_l}\left\{\partial_x F\circ \Omega\right\})\cdot D\Upsilon (\partial_u F^n)^{-1}\big]\\
&=\tr\left[(\partial_{s_l}\Lambda)\Lambda^{-1}\right]+\sum_k\partial_{s_l}\Phi_k\left\{\tr\left[(\partial_{y_k}\partial_x F(\partial_xF)^{-1}\right]\right\} \circ \Omega\\
&=H^F\circ \Omega\cdot\partial_{s_l}\Phi+\tr\left[(\partial_{s}\Lambda)\Lambda^{-1}\right].
\end{split}
\end{equation}
Moreover, note that $\tr\left[(\partial_{s_l}\Lambda)\Lambda^{-1}\right](0)=\sum_k\{\tr[\partial_{y_k}A\cdot A^{-1}] \cdot[E^{-1}]_{kl}\}(0)$.
Hence, using \eqref{eq:Phis} we obtain, for all $n\in\{n_\star,\dots, 2 n_\star\}$,
\begin{equation}\label{eq:h-abs}
\|H^{F^n}\|_{\cC^0}\leq \|H^{F}\cdot\partial_{s_l}\Phi\|_{\cC^0}+c_n  \leq \frac{ \Czero}{1-\srk}\nu^n \|H^{F}\|_{\cC^0}+c_n \leq \Lb_1/2
\end{equation}
provided $\Lb_1$ is large enough.
Differentiating \eqref{eq:hol-h} yields, for each $0<l\leq r-2$,
\begin{equation}\label{eq:exH}
\begin{split}
\partial_{s_{j_l}}& \cdots \partial_{s_{j_1}}H^{F^n}=\sum_{i_1,\dots, i_l}\left[\partial_{y_{i_l}}\cdots \partial_{y_{i_1}} H^{F}\right]\circ\Omega\cdot \partial_{s} \Phi \cdot\partial_{s_{j_1}}\Phi_{s_{i_1}}\cdots \partial_{s_{j_l}}\Phi_{i_l} \\
&+\overline\Theta_{l}(DS\circ \bF\circ \Omega, \dots, D^{l+2}S\circ  \bF\circ \Omega, \partial_s^{l+1} F^n, \dots, \partial_sF^n, 
\partial_y^{l} F\circ\Omega, \dots\\
&\phantom{+\overline\Theta_{l}(\ }
\dots \partial_yF\circ \Omega, H^{F}\circ \Omega, \dots, [\partial_y^{l-1}H^{F}]\circ\Omega),
\end{split}
\end{equation}
where $\overline\Theta_{l}$ is a sum of terms that either do not depend on $\partial_s^{p}H^{F^n}$, for all $p<l$, or are linear in a $\partial_s^{p}H^{F^n}$, for some $p<l$,  $k_{p,j}$-multilinear in $\partial_y^jF$,  for $j\in\{2,\dots, l+1\}$, and $q_{p,j}$ multilinear in $\partial_s^jF^n$, for $j\in\{2,\dots, l+2\}$,  such that\footnote{ We use the convention that $q_{p, l+2}=0$ and $\partial_s^{-1}H^{F^n}=1$.}
\[
\sup_{p\in\{-1,\dots l-1\}}\left[ p+\sum_{j=2}^{l+2} (k_{p,j}+q_{p,j})(j-1)\right]\leq l.
\]
Let us verify it: equation \eqref{eq:hol-h} shows that it is true for $l=0$. Let us assume it true for $l-1$, then differentiating the first term we obtain the correct term linear in $\partial^l H^F$, the other terms are linear in $\partial^{l-1} H^F$ and linear in $\partial^2F\circ \Omega$ (see equation \eqref{eq:s-der}) hence $p'=l-1$, $k_{l-1,2}=1$ and all the other degree are zero, so $p'+k_{l-1,2}\leq l$. Differentiating $\overline\Theta_{l-1}$ with respect to $D^mS\circ \bF\circ \Omega$ does not change the multilinearity indexes. Differentiating with respect to $\partial_s^j F^n$ yields, for each $p$, a term with $p'=p$, $q_{p',j}'=q_{p,j}-1$ multilinear in $\partial_s^j F^n$ and $q_{p',j+1}'=q_{p,j+1}+1$ multilinear in $\partial_s^{j+1} F^n$. Thus $p'+\sum_{j=2}^{l+2} (k'_{p',j}+q'_{p',j})(j-1)\leq l$. The same happens if one differentiates with respect to $\partial_y^j F\circ \Omega$ for $j\geq 2$. On the other hand differentiating with respect to 
$\partial_y F\circ \Omega$ yields a term in which $p'=p$, $k_{p',2'}=k_{p,2}+1$, thus $p'+\sum_{j=2}^{l+2} (k'_{p',j}+q'_{p',j})(j-1)\leq l$. 
Finally, if we differentiate with respect to $\partial^j_y H^F\circ \Omega$ for $0\leq j<l-1$ we have a term with $p'=p+1$ and $k'_{p',j}=k_{p,j}$, $q'_{p',j}=q_{p,j}$, thus, again $p'+\sum_{j=2}^{l+2} (k'_{p',j}+q'_{p',j})(j-1)\leq l$.
Which proves the claim.

Remembering \eqref{eq:Phis}, definition \eqref{eq:holder-fol} and equation \eqref{eq:Fnbound},  it follows that, for all $l\in\{1,\dots, r-2\}$, 
\begin{equation}\label{eq:HFder}
\begin{split}
\|\partial^l_{s} &H^{F^n}\|\leq 8^{-{r(l+1)}}L_1^{(l+1)^2}+\sup_{p\in \{-1,\dots, l-1\}} \Const L_1^{(p+1)^2+\sum_{j=2}^{l+2} (k_{p,j}+q_{p,j})(j-1)^2}\\
&\leq 8^{-{(l+1)^2}}L_1^{(l+1)^2}\!+\!\!\sup_{p\in \{-1,\dots, l-1\}}\!\!\! \Const L_1^{(p+1)^2+\left [\sum_{j=2}^{l+2} (k_{p,j}+q_{p,j})(j-1)\right](l+1)}\\
&\leq 8^{-{(l+1)^2}}L_1^{(l+1)^2}+\sup_{p\in \{-1,\dots, l-1\}} \Const L_1^{(p+1)^2+(l-p)(l+1)}\\
&\leq \left[8^{-{(l+1)^2}}-\Const L_1^{-l}\right] L_1^{(l+1)^2}\leq {\left(\frac{L_1}2\right)^{(l+1)^2}}
\end{split}
\end{equation}
provided $L_1$ is chosen large enough. 

To prove the bound on the H\"older semi-norm we use \eqref{eq:exH}, \eqref{eq:phi-hold} and proceed as in \eqref{eq:Fholder}: for each $l\leq r-3$
\[
\begin{split}
\|\partial^l_sH^{F^n}(a)-\partial^l_sH^{F^n}(b)\|\leq& \|\partial_y^l H^F\circ\Omega(a)-\partial_y^l H^F\circ\Omega(b)\| C_\star^{l+1}\nu^{(l+1)n}\\
&+\Const L_1^{(l+1)^2+1}\|u\|^{\tau_0}+\|\overline\Theta_l(a)-\overline\Theta_l(b)\|,
\end{split}
\]
where we have used \eqref{eq:Phis} and \eqref{eq:phi-hold}, \eqref{eq:almost-last}. Next we use Lemma \ref{lem:aux}, with $\cG=\partial^l_yH^F$, $D_0=L_1^{(l+1)^2}$, $D'=\tilde D=L_1^{(l+2)^2}$ and $\tau=\tau_0$, to write
\[
\begin{split}
\|\partial^l_sH^{F^n}(a)-\partial^l_sH^{F^n}(b)\|\leq& C_l\nu^{l n} L_1^{(l+2)^2}\|u\|^{\tau_0}\\
&+\Const L_1^{(l+1)^2+1}\|u\|^{\tau_0}+\|\overline\Theta_l(a)-\overline\Theta_l(b)\|
\end{split}
\]
The claim follows then by induction and using the known structure of $\overline\Theta_l$.
\end{proof}
We conclude the section by clarifying the relation between the function $H^F$ and the holonomy associated with the foliation $\bF$.
 The next Lemma shows that the Jacobian of the Holonomy can be seen as a flow of which $H^F$ is the ``generator".
\begin{lem}\label{lem:holo}
If $W\in \cW_L^{r,0}$, then there exists $C>0$ and $\rho_0>0$ such that, for each $\xi\in  M$, $0<\rho<\rho_0$ and $\|(x',y')\|\leq \rho$, we have $\|\det(\partial_x F_\xi)(x',\cdot)\|_{\cC^q}\leq C$. More precisely, setting $J^F_\xi(x,y)=\det(\partial_x F_\xi)(x,y)$, we have
\begin{equation}\label{eq:det}
\begin{split}
&\partial_{y}J^F_\xi=J^F_\xi\cdot H^F_\xi\circ \bF\\
&J^F_\xi(x,0)=1.
\end{split}
\end{equation}
\end{lem}
\begin{proof}
Let $(x',y')$ as in the Lemma's assumption.
First of all note that for each vector $e_i\in\bR^d$ we have
\[
\begin{split}
\partial_{y_i}\det(\partial_x F_\xi)(x',y')&=\det(\partial_x F_\xi)\lim_{h\to 0}\frac{\det\big((\partial_x F_\xi)(x',y')^{-1}\cdot\partial_x F_\xi(x',y'+he_i)\big)-1}h\\
&=\det(\partial_x F_\xi)\lim_{h\to 0}\frac{e^{\tr\big(\ln (\Id+(\partial_x F_\xi)(x',y')^{-1}\partial_{y_i}\partial_x F_\xi(x',y') h)\big)}-1}h\\
&=\det(\partial_x F_\xi)(x',y')\tr\big((\partial_x F_\xi)(x',y')^{-1}\partial_x (\partial_{y_i}F_\xi)(x',y')\big).
\end{split}
\]
Thus
\[
\partial_{y}\det(\partial_x F_\xi)(x',y')=\det(\partial_x F_\xi)(x',y')\cdot H^F_\xi\circ \bF(x',y')
\]
which immediately implies the Lemma since $\det(\partial_x F_\xi)(x',0)=1$ by construction.
\end{proof}
\begin{rem}\label{rem:holo}
 Lemma \ref{lem:holo} implies that, for each measurable set $B\subset \bR^{d_u}$ and $|\axp|\leq r-1$, holds
 \[
\left|\partial^\axp_y|F(B,y)|\right|= \left|\partial^\axp_y\int_{F(B,y)} dx\right|=\left|\int_B \partial^\axp_y\det(\partial_x F)(x,y) dx\right|\leq C|B|.
 \]
 Note that the first and last term of the above inequality do not involve $\partial_x F$, hence it holds also for $F$ non differentiable with respect to $x$, provided they are the limit of foliations $F_k$ (in the sense that the $\partial^\axp_yF_k$ converge) that satisfy the inequality uniformly. The same Remark holds also for equation \eqref{eq:det}. In other words if we consider the true invariant foliation, where $\partial_x F$ may make no sense, still $H^F$ is well defined (see Remark \ref{rem:invariant} for details), and so, by equation \eqref{eq:det}, is the Jacobian of the Holonomy $J^F$. 
 \end{rem}
 \section{ Test Functions}\label{sec:test}
\begin{proof}[Proof of Lemma \ref{lem:dual}]
By equation \eqref{eq:f-transf} it follows

\begin{equation}\label{eq:vf-est}
\begin{split}
\vf\circ T^n\circ \phi_i^{-1}\circ \bF^n(u,s)&=\vf\circ \phi_j^{-1}\circ S^{-1}\circ \bF^n(u,s)=\vf\circ \phi_j^{-1}\circ\bF(\Omega(u,s))\\
&=\vf\circ \phi_j^{-1}\circ\bF(\Upsilon(u),\Phi(u,s)).
\end{split}
\end{equation}
Accordingly, $\|\vf\circ T^n\|_{\cC^0}^{T^{-n}W}\leq \|\vf\|_{\cC^0}^{W}$ and
\[
\partial_s\left[\vf\circ T^n\circ \phi_i^{-1}\circ \bF^n\right]\!(u,s)= \sum_l\partial_{z_l}\left[\vf\circ \phi_j^{-1}\circ\bF\right]\!(\Omega(u,s))\partial_{s}\Phi_l(u,s).
\]
Then, differentiating further the above computation yields, for some $C_*>0$,
\begin{equation}\label{eq:dcontract}
\param\left|\partial_s^q\left[\vf\circ T^n\circ \phi_i^{-1}\circ \bF^n\right]\!(u,s)\right|\leq \|\vf\|^W_q \|\partial_{ s}\Phi_l\|_{\cC^0}^q+\param^{-1}C_*\|\vf\|^W_{q-1}.
\end{equation}
By \eqref{eq:Phis} $\|\partial_s\Phi\|_{\cC^0}\leq \Const \sigma^n$, while there exists $A_0>1$ such that $\|\partial_s^i\Phi \|_{\cC^0}\leq A_0$ for all $i\leq r$.

From this and recalling the definition \eqref{def of the C norm}  follows
 \[
 \|\vf\circ T^n\|_{q}^{T^{-n}W}  \leq A_0 \|\vf\|_{q}^{W}.
\]
 On the other hand, recalling \eqref{eq:norm-q-def}, there exists $B_0>0$ such that
\[
 \begin{split}
 \|\vf\circ T^n\|_{q+1}^{T^{-n}W} & =  \|\partial_s \vf\circ T^n\|_{q}^{T^{-n}W}+ \param^{q+1}\|\vf\circ T^n\|_{C^0}^{T^{-n}W}\\
 &\leq A_0\sigma^{qn} \|\vf\|_{q+1}^{W}+( 1+\param^{-1}C_*) \|\vf\|_{\cC^q}^W \\
 &\leq A_0\sigma^{qn} \|\vf\|_{q+1}^{W}+B_0\|\vf\|_{\cC^q}^{W}.
\end{split}
\]
\end{proof}

\end{document}